\documentclass[leqno,12]{amsart}

\usepackage[dvipsnames]{xcolor}

\usepackage{amssymb,enumitem,tikz,tabularx,cancel,calc,caption,diagrams}
\usetikzlibrary{shapes,snakes}
\usepackage{ytableau}
\usepackage[vcentermath]{youngtab}
\usepackage[
	bookmarksnumbered,
	pdftitle={{G}rothendieck polynomials and alternating signs via iterated residues}
	pdfauthor={Justin Allman}
	pdfsubject={2010 MSC Classification: 05E05, 14N15}
	]
	{hyperref}%hyperref package

\usepackage{geometry}
\DeclareCaptionFormat{myfigformat}{#1#2#3\hrulefill}
\DeclareCaptionFormat{mytabformat}{~ \\[10pt] #1#2#3\hrulefill}
\newcommand{\curly}{\mathcal}
\newcommand{\GG}{\curly G}
\newcommand{\SSS}{\curly S}
\newcommand{\HH}{\curly H}
\newcommand{\DD}{\curly D}

\newcommand{\RR}{\curly R}
\DeclareMathOperator{\Res}{Res}
\newcommand{\IR}{\mathop{\Res}} %for Residue operations with subscripts

\newcommand{\T}[3]{\mathbb{T}_{#1,#2}^{#3}}
\DeclareMathOperator{\Bad}{Bad}
\DeclareMathOperator{\Seg}{Seg}
\DeclareMathOperator{\Max}{Max}
\DeclareMathOperator{\code}{Code}
\DeclareMathOperator{\ocode}{OddCode}
\DeclareMathOperator{\ecode}{EvenCode}
\newcommand{\Comp}{\curly C}

\newcommand{\bfs}[1]{\boldsymbol{#1}}
\newcommand{\bfst}{{\bfs{t}}}

\newcommand{\Z}{\mathbb{Z}}%integers
\newcommand{\N}{\mathbb{N}}%natural numbers

\newcommand{\disjointunion}{\sqcup}
\newcommand{\DisjointUnion}{\bigsqcup}
\newcommand{\intersect}{\cap}

\newcommand{\Dirsum}{\bigoplus}

\theoremstyle{plain}
\newtheorem{prop}{Proposition}[section]
\newtheorem{lem}[prop]{Lemma}
\newtheorem{conj}[prop]{Conjecture}
\newtheorem{cor}[prop]{Corollary}
\newtheorem{thm}[prop]{Theorem}

\theoremstyle{definition}
\newtheorem{defn}[prop]{Definition}

\theoremstyle{remark}
\newtheorem{remark}[prop]{Remark}
\newtheorem{example}[prop]{Example}

\title[Iterated residue perspective on stable {G}rothendieck polynomials]{An iterated residue perspective on stable {G}rothendieck polynomials}

\author[Allman]{Justin Allman}
\address{Department of Mathematics, U.S.~Naval Academy \\ Annapolis, MD}
\email{allman@usna.edu}

\author[Rim{\'a}nyi]{Rich\'ard Rim\'anyi}
\address{Department of Mathematics, UNC--Chapel Hill \\ Phillips Hall CB\#3250 \\ Chapel Hill, NC 27599--3250}
\email{rimanyi@email.unc.edu}


\begin{document}

%%%%%%%%%%%%%%
%% ABSTRACT %%
%%%%%%%%%%%%%%

\begin{abstract}
Grothendieck polynomials are important objects in the study of the $K$-theory of flag varieties. Their many remarkable properties have been studied in the context of algebraic geometry and tableaux combinatorics. We explore a new tool, similar to generating sequences, which we call the iterated residue technique. We prove new formulas on the calculus of iterated residues and use them to prove straightening laws and multiplication formulas for stable Grothendieck polynomials. As a further application of our method, we give new proofs that the $K$-Pieri rule and the expansions of Grothendieck polynomials in the Schur basis both exhibit alternating signs. As a consequence, we observe that our method implies a new combinatorial statement of the $K$-Pieri rule. Our results indicate that the iterated residue technique should be further explored as a new line of attack on open conjectures regarding positivity and stability, e.g.~of quiver polynomials and Thom polynomials, in $K$-theory.
\end{abstract}

\maketitle

\tableofcontents

%%%%%%%%%%%%%%%%%%%%%%%%%%
%% INTRODUCTORY SECTION %%
%%%%%%%%%%%%%%%%%%%%%%%%%%

\section{Introduction}

Grothendieck polynomials, one for each permutation, were introduced by Lascoux and Sch\"utzenberger in their seminal work \cite{alms1982} as representatives for $K$-classes of Schubert varieties in flag manifolds. They play the role in $K$-theory of the Schubert polynomials in cohomology. Stable versions of these polynomials, again corresponding to any permutation, appeared in the works of Fomin and Kirillov \cite{sfak1994,sfak1996} and were shown to be supersymmetric via an analogous construction to that of the Stanley symmetric functions. Buch showed that the stable Grothendieck polynomials corresponding to Grassmannian permutations, and therefore to partitions, form a $\Z$-linear basis for all stable Grothendieck polynomials \cite{ab2002.klr}. It is these polynomials which are the subject of this paper.

Over the past decade the stable Grothendieck polynomials corresponding to partitions have become important to many problems in the realm of algebraic combinatorics, for example in \cite{cl2000,ab2002.klr}, and in particular have appeared in several algebro-geometric contexts:
\begin{itemize}
\item as representatives of Schubert varieties in the $K$-theory of Grassmannians, e.g.\ in \cite{mb2002,ab2005.ckt},
\item as the proper basis to describe $K$-classes of degeneracy loci for Dynkin quivers, e.g.\ in \cite{ja2014.ir,ab2002.qv,ab2008,em2005}, and
\item as a basis for Thom polynomials in $K$-theory \cite{rras2017}.
\end{itemize}
In each of these settings (and more) alternating signs have been proven or at least conjectured to appear in the relevant formulae, see \cite{mb2002,ab2002.qv,ab2002.klr,ab2008,cl2000,em2005,rras2017} and many more. These results regarding alternating signs appear to depend on an amalgam of geometric techniques e.g.\ in \cite{mb2002}, combinatorics e.g.\ in \cite{ab2002.qv,ab2002.klr,ab2008,cl2000,em2005}, or identities in the theory of symmetric functions e.g.\ in \cite{cl2000}.

A formula for the stable Grothendieck polynomials in terms of iterated residues was established in \cite{rras2017} and was used in \cite{ja2014.ir} to expand quiver polynomials in $K$-theory. In this paper, we apply the iterated residue technique to
	\begin{enumerate}[label=(\arabic*)]
	\item establish independent proofs of straightening laws for expanding Grothendieck polynomials corresponding to non-partition integer sequences into $\Z$-linear combinations of those for partitions (Section \ref{ss:G.Straighten});
	\item multiply stable Grothendieck polynomials, in particular, to prove the finiteness and positivity of the $K$-theoretic Pieri rule (Sections \ref{s:GG.ratl}, \ref{s:K.Pieri.finiteness}, \ref{s:K.Pieri.positivity}, and \ref{s:translate.comb});
	\item conjecture an iterated residue formula for the general case of  multiplication (Section \ref{ss:q.greater.1}); and
	\item expand single stable Grothendieck polynomials into the Schur basis of symmetric functions and prove that this expansion has alternating signs (Section \ref{s:G.poly.S.expand}).
	\end{enumerate}

\subsection*{Acknowledgements}
{We would like to thank A.~Szenes for helpful discussions on iterated residues, and an anonymous referee for suggestions on (a) the definition of the $\GG_\bfst$ operation applied to a rational function and (b) the more aesthetically pleasing form of Corollary \ref{cor:GG.K.Pieri.comb.Asets} \emph{vis-{\`a}-vis} Theorem \ref{thm:GG.K.Pieri.comb}.} 
A portion of this work has been submitted as an extended abstract to the Formal Power Series in Algebraic Combinatorics conference, to be held July 2018 in Hanover, New Hampshire, USA.

\section{Iterated residues and Grothendieck polynomials}
\label{s:Res.Gpolys.Gamma}

\subsection{Notation}
\label{ss:notation}

Consider a sequence $I = (I_1,I_2,\ldots)$ of integers. Each integer $I_i$ is called a \emph{part} of $I$. The sequence $I$ is called \emph{finite} if only finitely many of its parts are nonzero. For a finite integer sequence $I$, we define its \emph{length} $\ell(I) = \max\{ i  : I_i \neq 0 \}$ and its \emph{weight} $|I| = \sum_i I_i$. A finite integer sequence $\lambda$ is called a \emph{partition} if $\lambda_1 \geq \lambda_2 \geq \cdots \geq \lambda_{\ell(\lambda)} > 0$. Throughout, we identify the partition $\lambda$ with its Young diagram of boxes, e.g.
\[
\lambda = (4,2,1) \leftrightsquigarrow  
\ytableausetup{centertableaux,smalltableaux}
{\tiny
\ydiagram{4,2,1}
}.
\]
If the Young diagram for $\mu$ contains $\lambda$ as a subdiagram, we write $\mu\supset\lambda$, and in this case we let $\mu/\lambda$ denote the resulting skew diagram, which is called a \emph{horizontal $n$-strip} if $|\mu|-|\lambda|=n$ and no two boxes of $\mu/\lambda$ appear in the same column.

Given two finite integer sequences $I$ and $J$ of respective lengths $k$ and $l$ we define the sequence $I,J$ to be the concatenation $(I_1,\ldots,I_k,J_1,\ldots,J_l)$ and the sequence $I+J$ to be $(I_1+J_1,\ldots,I_m+J_m)$ where $m=\max\{k,l\}$ and it is understood that $I_i = 0$ for $i>k$ and likewise $J_j=0$ for $j>l$. Similarly, we can define $I-J$. Moreover, we let $\N = \{0,1,2,\ldots\}$ and for any positive integer $n$, we set $[n] = \{1,2,\ldots,n\}$. 

For variables $\bfs{u} = \{u_1,u_2,\ldots\}$, we let $\Z[\bfs{u}]$ and $\Z[\bfs{u}^{\pm1}]$ respectively denote the ring of polynomials and Laurent polynomials in $\bfs{u}$. Similarly, $\Z(\bfs{u})$ denotes the ring of rational functions in $\bfs{u}$. We will also denote the concatenation of finite lists of variables $\bfs{u} = \{u_1,\ldots,u_k\}$ and $\bfs{v}=\{v_1,\ldots,v_l\}$ with $\bfs{u},\bfs{v} = \{u_1,\ldots,u_k,v_1,\ldots,v_l\}$.

\subsection{Iterated residues}
\label{ss:Residues}

Given a meromorphic function $\phi(z)$ in the single complex indeterminate $z$, we define the operation
\begin{equation}
\label{eqn:Res.1var}
\IR_{z=0,\infty}(\phi(z)\,dz) = \IR_{z=0}(\phi(z)\,dz) + \IR_{z=\infty}(\phi(z)\,dz)
\end{equation}
where we recall that $\IR_{z=\infty}(\phi(z)\,dz) = \IR_{z=0}(-z^{-2}\phi(1/z)\,dz)$. Now, for $f(\bfs{z})$ a meromorphic function in $\bfs{z} = (z_1,\ldots,z_p)$ we define
\begin{equation}
\label{eqn:Res.multvar}
\IR_{\bfs{z}=0,\infty}\left(f(\bfs{z})\,dz_p\cdots dz_1 \right) = 
	\IR_{z_1=0,\infty} \cdots  \IR_{z_p=0,\infty} \left(f(\bfs{z})\,dz_p\cdots dz_1 \right).
\end{equation}
%In fact, for any permutation $w\in S_p$, we define
%\begin{equation}
%\label{eqn:Res.multvar.perm}
%\IR_{\bfs{z}=0,\infty}^w \left(f(\bfs{z})\,dz_{w(p)}\cdots dz_{w(1)} \right) = 
%	\IR_{z_{w(1)}=0,\infty} \cdots  \IR_{z_{w(p)}=0,\infty} \left(f(\bfs{z})\,dz_{w(p)}\cdots dz_{w(1)} \right).
%\end{equation}
%When the operator $\IR$ is not decorated with a superscript, as in Equation \eqref{eqn:Res.multvar}, it is understood that $w$ is the identity. Moreover, throughout the sequel we let $\tau_i \in S_p$ denote the simple transposition interchanging $i$ and $i+1$.

\subsection{Stable {G}rothendieck polynomials}
\label{ss:Gpolys}

Let $p$ be a positive integer, $I = (I_1,\ldots,I_p)$ a sequence of integers, and $\bfs{z} = \{z_1,z_2,\ldots,z_p\}$ be a set of complex-valued indeterminates. We need as many variables $z_i$ as there are entries in $I$. Let $\bfs{\alpha}=\{\alpha_1,\alpha_2,\ldots,\alpha_k\}$ and $\bfs{\beta}=\{\beta_1,\beta_2,\ldots,\beta_l\}$ be sets of variables. We now define a (Laurent) polynomial $G_I(\bfs{\alpha};\bfs{\beta})$, actually a polynomial in the variables $\alpha_i$ and $\beta_j^{-1}$, as follows. We set
\begin{gather}
\Delta(\bfs{z}) = \prod_{1\leq i<j\leq p} \left(1 - \frac{z_j}{z_i} \right)
	\label{eqn:Delta} \\
P(\bfs{z} | \bfs{\alpha};\bfs{\beta}) 
	= \prod_{i=1}^p (1-z_i)^{k-l} \frac{\prod_{b=1}^l (1-z_i \beta_b)}{\prod_{a=1}^k (1-z_i\alpha_a)}
	\label{eqn:P}
\end{gather}
and define
\begin{equation}
\label{eqn:defn.G}
G_I(\bfs{\alpha};\bfs{\beta}) = 
	\IR_{\bfs{z} = 0,\infty}\left(
		\frac{\prod_{i=1}^p (1-z_i)^{I_i - i}}{z_1\cdots z_p} \, P(\bfs{z}|\bfs{\alpha};\bfs\beta) \, \Delta(\bfs{z}) \,dz_p\cdots dz_1\right)
\end{equation}
which we call a \emph{double stable Grothendieck polynomial}. Remarkably, the second author and Szenes have proven that when $I$ is a partition, i.e.~$I_1\geq I_2 \geq \cdots \geq I_p>0$, the formula above agrees with previous formulations of the polynomial $G_I$ \cite{rras2017}. In particular, $G_I$ is \emph{supersymmetric} in the variables $\bfs\alpha$ and $\bfs\beta$; i.e.\ it is separately symmetric in each set of variables. For two such notable formulations of Grothendieck polynomials, we refer the reader to the original definition of Lascoux--Sch\"utzenberger in terms of divided difference operators \cite{alms1982} and the combinatorial set-valued tableaux description of Buch \cite{ab2002.klr}. 

We remark that the seminal Lascoux--Sch\"utzenberger paper \cite{alms1982} actually defines a (non-stable) Grothendieck polynomial for any permutation; our description \eqref{eqn:defn.G} applies only to Grothendieck polynomials for so-called Grassmannian permutations (which are in one-to-one correspondence with partitions). Fomin and Kirollov defined the stable versions of these polynomials and proved that they are supersymmetric \cite{sfak1994,sfak1996}. The relationship of the stable polynomials $G_I$ above to the non-stable double Grothendieck polynomials is analogous to the relationship between Stanley symmetric functions (stable and supersymmetric) and double Schubert polynomials. 

However, since the (non-stable) Grothendieck polynomials corresponding to Grassmannian permutations are already separately symmetric in the $\bfs\alpha$ and $\bfs\beta$ variables, they coincide with the stable versions evaluated on finitely many variables. Moreover, Buch proved \cite[Theorem~4]{ab2002.qv} that \emph{every} stable Grothendieck polynomial can be expressed as a polynomial in those corresponding to partitions (or equivalently Grassmannian permutations).

\begin{remark}
\label{rem:Fubini}
Since the meromorphic function of which we compute residues in Equation \eqref{eqn:defn.G} does not blow up along any hyperplane $z_i = z_j$, Fubini's Theorem implies that the residues can be taken in any order. As a consequence of this fact, we obtain that for any permutation $w$, applying our residue operation to 
\[
\frac{\prod_{i=1}^p (1-z_{w(i)})^{I_i - i}}{z_{w(1)} \cdots z_{w(p)}} \, P(w\cdot\bfs{z}|\bfs{\alpha};\bfs\beta) \, \Delta(w\cdot\bfs{z})
= 
\frac{\prod_{i=1}^p (1-z_{w(i)})^{I_i - i}}{z_1 \cdots z_p} \, P(\bfs{z}|\bfs{\alpha};\bfs\beta) \, \Delta(w\cdot\bfs{z}) 
\]
also evaluates to $G_I(\bfs{\alpha};\bfs\beta)$. We have used the fact that $P$ and the monomial $z_1\cdots z_p$ are both symmetric in the $\bfs z$ variables. We will use this observation in the sequel.
\end{remark}

\begin{remark}
\label{rem:suppress.alpha.beta}
Whenever a result about stable Grothendieck polynomials is true regardless of the size $k$ and $l$ of the sets of variables $\bfs\alpha$ and $\bfs\beta$, as is the case for Equation \eqref{eqn:G.I.to.sum.lambda} below (this will be evident from our method), then we may suppress $\bfs\alpha$ and $\bfs\beta$ in our notation.
\end{remark}

Let $\Gamma$ denote the $\Z$-linear span $\Dirsum_\lambda \Z\cdot G_\lambda$ taken over all partitions $\lambda$. In fact, $\Gamma$ forms a $\Z$-algebra with multiplication $G_\lambda \, G_\mu = \sum_{\nu} c_{\lambda,\mu}^\nu G_\nu$. The structure constants $c_{\lambda,\mu}^\nu$, only finitely many of which are nonzero for fixed $\lambda$ and $\mu$, are the \emph{$K$-theoretic Littlewood--Richardson numbers} and can be computed by generalized tableaux combinatorics, a result of Buch \cite{ab2002.klr}. We will investigate the coefficients $c_{\lambda,(n)}^\nu$ for $n$ a positive integer in the sequel.

\subsection{An application: straightening laws}
\label{ss:G.Straighten}

In the context of studying quiver polynomials, Buch first introduced stable Grothendieck polynomials corresponding to integer sequences, not necessarily partitions \cite[Section~3]{ab2002.qv}. In \emph{op.~cit.} the polynomial $G_I(\bfs\alpha;0)$, a.k.a.~the \emph{single} Grothendieck polynomial, is defined by a determinant and furthermore, by switching rows in this determinant, a formula to ``straighten'' $G_I$ as $\Z$-linear combination of $G_\lambda$ for partitions $\lambda$ was proven. 

\begin{thm}[c.f.\ \cite{ab2002.qv}, Corollary~3.3]
\label{thm:G.I.in.Gamma}
For every integer sequence $I$, the polynomial $G_I$ defines a unique element of $\Gamma$; i.e.~there exist unique integers ${d_\lambda^I}$, only finitely many of which are nonzero, such that 
\begin{equation}
\label{eqn:G.I.to.sum.lambda}
G_I = \sum_{\mathrm{partitions}~\lambda} 
				d_{\lambda}^I \, G_\lambda \in \Gamma.
\end{equation}
\end{thm}

However, the size of the determinant describing $G_I(\bfs\alpha;0)$ in \cite{ab2002.qv} depends on the number of $\bfs\alpha$ variables. We note that our formula \eqref{eqn:defn.G} defines the \emph{double} Grothendieck polynomial $G_I(\bfs\alpha;\bfs\beta)$ independently of the number of $\bfs\alpha$ and/or $\bfs\beta$ variables; its complexity is related instead to the length of $I$. This is more similar to determinantal formulas for Schur functions, e.g.~the Jacobi--Trudi formula \cite[Ch.\ I, (3.4)]{im1995}. We will reiterate this point in Remark \ref{rem:GI.det.vs.GI.IR}.

The goal of this subsection is to give an iterated residue proof of two straightening rules (see \eqref{eqn:G.Straighten} and \eqref{eqn:drop.nonpos} below), which together imply a new proof of Theorem \ref{thm:G.I.in.Gamma} which at once 
\begin{enumerate}[label=(\alph*),leftmargin=*]
\item is independent of the number of $\bfs\alpha$ variables, and
\item already includes the second set of variables $\bfs\beta$.
\end{enumerate}

\begin{thm}[c.f.~\cite{ab2002.qv}, Equation~(3.1)]
\label{thm:G.Straighten}
For any integer sequences $I$ and $J$, and any integers $a$ and $b$, we have the following ``straightening law'':
\begin{equation}
\label{eqn:G.Straighten}
G_{I,a,b,J} - G_{I,a+1,b,J} = G_{I,b,a+1,J} - G_{I,b-1,a+1,J}.
\end{equation}
Furthermore, if $J$ has only non-positive parts, then we have
\begin{equation}
\label{eqn:drop.nonpos}
G_{I,J} = G_I.
\end{equation}
\end{thm}

\begin{proof}
We prove \eqref{eqn:G.Straighten} in the case that $I$ and $J$ are empty. The general case is analogous, only with more notation. For the righthand side of \eqref{eqn:G.Straighten}, we consider the result of applying $\Res_{\bfs z =0,\infty}$ to
\begin{equation}
\label{eqn:G.Straighten.proof.1}
\left[(1-z_1)^{b-1}(1-z_2)^{a-1} - (1-z_1)^{b-2}(1-z_2)^{a-1} \right] \left(1- \frac{z_2}{z_1} \right) \frac{P(\bfs z|\bfs\alpha;\bfs\beta)}{z_1z_2} \,dz_2 dz_1.
\end{equation}
Using the observation of Remark \ref{rem:Fubini} we apply the simple transposition $z_1\leftrightarrow z_2$ and Fubini's theorem to obtain
\begin{equation}
\label{eqn:G.Straighten.proof.2}
\mathop{\underbrace{\mathop{\left[(1-z_1)^{a-1}(1-z_2)^{b-2} - (1-z_1)^{a-1}(1-z_2)^{b-1}\right] \left(\frac{z_1}{z_2}\right)}}}_{\clubsuit} \left(1- \frac{z_2}{z_1} \right) \frac{P(\bfs z|\bfs\alpha;\bfs\beta)}{z_1z_2}\,dz_2 dz_1.
\end{equation}
The portion of the above expression labeled $\clubsuit$ is equal to
\[
(1-z_1)^{a-1}(1-z_2)^{b-2} - (1-z_1)^{a+1-1}(1-z_2)^{b-2}
\]
and hence the result of applying $\Res_{\bfs{z}=0,\infty}$ to \eqref{eqn:G.Straighten.proof.2} is $G_{a,b} - G_{a+1,b}$ as desired. We prove \eqref{eqn:drop.nonpos} in the case that $\ell(J) = 1$; the general result follows inductively. Write $J = (j)$ with $j\leq 0$ and assume that $\ell(I) = p$. Set $\bfs{z} = (z_1,\ldots,z_p)$ and apply Fubini's theorem to get
\begin{multline}
\label{eqn:drop.nonpos.proof.1}
G_{I,j}(\bfs\alpha;\bfs\beta) = \IR_{\bfs{z}=0,\infty} \left(
	\frac{\prod_{i=1}^p (1-z_i)^{I_i-i}}{z_1\cdots z_p} 
		P(\bfs z | \bfs\alpha;\bfs\beta) \Delta(\bfs z)\, 
		dz_p\cdots dz_1 \right) 
	\\
\times \IR_{\zeta=0,\infty} \left( 
	(1-\zeta)^{j-(p+1)} \cdot \frac{1}{\zeta} \cdot 
	P(\zeta | \bfs\alpha;\bfs\beta) \cdot 
	\prod_{i=1}^p \left(1 - \frac{\zeta}{z_i}\right) \, d\zeta \right).
\end{multline}
We concentrate on the residues corresponding to the $\zeta$ variable. Set
\[
g(\zeta) = (1-\zeta)^{j-(p+1)} \cdot \frac{1}{\zeta} \cdot 
	P(\zeta | \bfs\alpha;\bfs\beta) \cdot 
	\prod_{i=1}^p \left(1 - \frac{\zeta}{z_i}\right)
\]
and note $\zeta=0$ is a simple pole of $g$. Thus, $\Res_{\zeta=0}(g\,d\zeta) = \lim_{\zeta\to0} (\zeta\cdot g) = 1$. 
To compute $\Res_{\zeta=\infty}(g\,d\zeta)$ we consider
\[
\widetilde{g}(\zeta) = -\frac{1}{\zeta^2}g(1/\zeta) = -\frac{1}{\zeta^2} \cdot (1-1/\zeta)^{j-p-1} \cdot \zeta \cdot P(1/\zeta|\bfs\alpha;\bfs\beta) \prod_{i=1}^p \left(1 - \frac{1}{\zeta z_i} \cdot \right)
\]
and use that $\Res_{\zeta=\infty}(g\,d\zeta) = \Res_{\zeta=0}(\widetilde{g}\,d\zeta)$. A calculation shows that
\[
\widetilde{g}(\zeta) = -\frac{1}{\zeta^j}\cdot(\zeta-1)^{j-p-1+k-l}\cdot\frac{\prod_{b=1}^l (\zeta-\beta_b)}{\prod_{a=1}^k (\zeta-\alpha_a)}\cdot\prod_{i=1}^p (\zeta-1/z_i).
\]
Since $j\leq 0$, we see that $\widetilde{g}$ is holomorphic at $\zeta=0$ and hence $\Res_{\zeta=0}(\widetilde g\,d\zeta) = 0$. This implies $\Res_{\zeta=0,\infty}(g\,d\zeta) = 1$ and hence Equation \eqref{eqn:drop.nonpos.proof.1} becomes
\[
G_{I,j}(\bfs\alpha;\bfs\beta) = \IR_{\bfs{z}=0,\infty} \left(
	\frac{\prod_{i=1}^p (1-z_i)^{I_i-i}}{z_1\cdots z_p} 
		P(\bfs z | \bfs\alpha;\bfs\beta) \Delta(\bfs z)\, 
		dz_p\cdots dz_1 \right) \cdot 1 = G_{I}(\bfs\alpha;\bfs\beta)
\]
as desired. Moreover, observe that proof has been independent of $k$ and $l$.
\end{proof}

\begin{proof}
By iteratively applying \eqref{eqn:G.Straighten} and \eqref{eqn:drop.nonpos} to $G_I$, we obtain an expansion of the form \eqref{eqn:G.I.to.sum.lambda}. The uniqueness of the expansion is implied by the fact that the polynomials $G_\lambda$ are linearly independent, say because their lowest degree homogeneous parts $s_\lambda$ are linearly independent.
\end{proof}

%In Sections \ref{s:K.Pieri.rule}, \ref{s:K.Pieri.finiteness}, \ref{s:K.Pieri.positivity}, and \ref{s:translate.comb} we consider the product $G_\lambda\,G_\mu$ in $\Gamma$ in the case that $\ell(\mu) = 1$. In Sections ?? we consider expansions of stable Grothendieck polynomials in the Schur basis.

\section{The operation $\GG_\bfst$ on Laurent polynomials}
\label{s:GG.Lpolys}

\begin{defn}
\label{defn:GG}
For any monomial $t^I$ we define 
\begin{equation}
\label{eqn:GG.defn}
\GG_{\bfst}(t^I) = G_I
\end{equation}
and extend the operation linearly to obtain a $\Z$-module mapping $\GG_{\bfst}:\Z[\bfst^{\pm1}] \to \Gamma$. If two Laurent polynomials $f_1$ and $f_2$ have the property that $\GG_\bfst(f_1 - f_2) = 0 \in \Gamma$, then we say that $f_1$ is \emph{$\GG_\bfst$-equivalent} to $f_2$.
\end{defn}

We now observe that the straightening laws \eqref{eqn:G.Straighten} and \eqref{eqn:drop.nonpos} can be encoded in the language of the $\GG_{\bfst}$ operation.

\begin{thm}[$\GG_{\bfst}$ version of Theorem \ref{thm:G.Straighten}]
\label{thm:GG.Straighten}
With $I$, $J$, $a$, $b$ as in Theorem \ref{thm:G.Straighten}, if $\ell(I)=k-1$ then we have
\begin{equation}
\label{eqn:GG.Straighten}
\GG_\bfst \left( 
t^{I,a,b,J}(1-t_{k})
\right) 
	= \GG_\bfst \left(
	- t^{I,b-1,a+1,J}(1-t_{k})
	\right). 
\end{equation}
Furthermore, if $f \in \Z[\bfst^{\pm1}]$ such that the exponent of every $t_i$ for $i>k$ is non-positive in every monomial term of $f$, then
\begin{equation}
\label{eqn:GG.drop.nonpos}
\GG_\bfst \left(f(\bfst)\right) = \GG_\bfst \left(f(t_1,\ldots,t_k,1,1,\ldots)\right).
\end{equation}
\end{thm}

\begin{proof}
Equations \eqref{eqn:GG.Straighten} and \eqref{eqn:GG.drop.nonpos} are respectively equivalent to Equations \eqref{eqn:G.Straighten} and \eqref{eqn:drop.nonpos}.
\end{proof}

\begin{cor}
\label{cor:GG.Straighten.symmetric}
Suppose that $f(\bfst)$ is a Laurent polynomial symmetric in $t_k$ and $t_{k+1}$. Then
\begin{equation}
\GG_\bfst \left( t^{I,a,b,J}(1-t_{k})\cdot f
\right) 
	= \GG_\bfst \left(
	- t^{I,b-1,a+1,J}(1-t_{k}) \cdot f
	\right).
\end{equation}
\end{cor}

\begin{proof}
By linearity, we assume that $f = t_k^c t_{k+1}^d + t_k^d t_{k+1}^c$ for some integers $c$ and $d$. We obtain
\[
t^{I,a,b,J} (1-t_k) f = t^{I,a+c,b+d,J}(1-t_k) + t^{I,a+d,b+c,J}(1-t_k).
\]
Apply \eqref{eqn:GG.Straighten} to both terms above to get the $\GG_\bfst$-equivalent polynomial
\[
-t^{I,b+d-1,a+c+1,J}(1-t_k) - t^{I,b+c-1,a+d+1,J}(1-t_k) = -t^{I,b-1,a+1,J}(1-t_k) f
\]
as desired.
\end{proof}

\subsection{More results on the calculus of iterated residues}
\label{ss:calculus}
The goal of the remainder of this section is to generalize the ``evaluate at $t=1$'' property evident in Equation \eqref{eqn:GG.drop.nonpos}. That is, we seek an elementary calculus to simplify monomials $t^I$ when the exponent of \emph{any} variable $t_i$ is too small. We begin with the following result.

\begin{cor}
\label{cor:binom.G.relation}
For any finite integer sequences $I$ and $J$, if $a>\ell(J)$ we have
\begin{equation}
\label{eqn:binom.G.relation}
\sum_{i=0}^{\ell(J)+1} (-1)^i \binom{\ell(J)+1}{i} G_{I,-a+i,J} = 0,
\end{equation}
or equivalently, with $\ell(I)=k-1$ we have
\begin{equation}
\label{eqn:binom.GG.relation}
\GG_\bfst \left( t^{I,-a,J}(1-t_{k})^{\ell(J)+1} \right) = 0.
\end{equation}
\end{cor}

\begin{proof}
We will use \eqref{eqn:GG.Straighten} to ``translate'' the negative entry $a$ to the right in the exponents of Equation \eqref{eqn:binom.GG.relation}. As such, the sequence $I$ will not effect our computation, and so we do the proof in the case that $\ell(I)=0$.

Write $l=\ell(J)$, $b=J_1$, and $\hat{J}=(J_2,\ldots,J_l)$; we seek to prove that 
\[\GG_\bfst\left(t^{-a,b,\hat{J}}(1-t_1)^{l+1}\right) = 0. \]
We will induct on $l$. In the case that $l=0$ we assume $a>0$ and hence
$\GG_\bfst\left(t_1^{-a}(1-t_1) \right) = 0$ by using Equation \eqref{eqn:GG.drop.nonpos} and setting $t_1=1$. For general $l$ we have
\[
t^{-a,b,\hat{J}}(1-t_1)^{l+1} 
= \left( \sum_{r=0}^l (-1)^l \binom{l}{r} t^{-a+r,b,\hat{J}} \right)(1-t_1)
\]
Hence, using Equation \eqref{eqn:GG.Straighten} on each term in the summation above yields the $\GG_\bfst$-equivalent polynomial
\begin{align*}
-\left(\sum_{r=0}^l (-1)^l \binom{l}{r}t^{b-1,-a+r+1,\hat{J}}\right)(1-t_1)
& = -t^{b-1,-a+1,\hat{J}}(1-t_2)^l\,(1-t_1) \\
& = t^{b,-(a-1),\hat{J}}(1-t_2)^l - t^{b-1,-(a-1),\hat{J}}(1-t_2)^l
\end{align*}
Now, since $\ell(\hat{J})=l-1$ and $a>l$ implies $a-1>l-1$, both terms in the last line are $\GG_\bfst$-equivalent to zero by induction.
\end{proof}

The observation that Grothendieck polynomial expressions can be simplified by algebraic operations on polynomials, i.e.~by specializing trailing variables to $1$ in Equation \eqref{eqn:GG.drop.nonpos}, has a generalization which we now describe. Given a function $f \in \Z(x)$ we let $T_{d,a}f(x)$ denote the degree $d$ Taylor polynomial for $f(x)$ centered at $a$.

\begin{defn}
\label{defn:DD}
Define the operation 
\[
\DD_k^{k+l}:\Z(t_1,\ldots,t_{k+l}) \to \Z(t_1,\ldots,\widehat{t_k},\ldots,t_{k+l}) [t_k^{\pm1}]
\] 
\begin{equation}
\label{eqn:DD.defn}
f \longmapsto \left.T_{l,1}\left(\left.f\right|_{t_k=x^{-1}}\right)(x)\right|_{x=t_k^{-1}}.
\end{equation}
Observe that $\DD_k^{k+l}$ is a composition of $\Z$-linear maps and therefore is linear.
\end{defn}

\begin{thm}
\label{thm:DD.GG.Equiv}
Let $I$, $J$ be finite integer sequences with $\ell(I)=k-1$, $\ell(J)=l$. If $a\geq l$, then 
\begin{equation}
\label{eqn:GG.DD}
\GG_{\bfst}\left(t^{I,-a,J}\right) = \GG_{\bfst} \left( \DD_k^{k+l}\left(t^{I,-a,J}\right) \right).
\end{equation}
\end{thm}

\begin{proof}
First we observe that 
\[
T_{l,1}(x^a) = \sum_{u=0}^l \binom{a}{u}(x-1)^u
\]
and hence
\begin{equation}
\label{eqn:DD.proof.1}
\DD_k^{k+l}\left(t^{I,-a,J}\right) = t^{I,0,J}\sum_{u=0}^l\binom{a}{u}\left(t_k^{-1} - 1\right)^u
\end{equation}
Now write $b=a-l$. Suppose that $b>0$ and we claim the following relation (which holds for integers $c$ with $0\leq c < b$):
\begin{equation}
\label{eqn:DD.proof.claim}
\GG_\bfst\left( t^{I,-c,J}\sum_{u=0}^l\binom{a-c}{u}\left(t_k^{-1} - 1\right)^u \right)
 = 
\GG_\bfst\left( t^{I,-c-1,J}\sum_{u=0}^l\binom{a-c-1}{u}\left(t_k^{-1} - 1\right)^u \right).
\end{equation}
We use that $\binom{a-c}{u} = \binom{a-c-1}{u} + \binom{a-c-1}{u-1}$ to rewrite the argument of the lefthand side in \eqref{eqn:DD.proof.claim} as
\begin{align*}
& t^{I,-c,J}\sum_{u=0}^l\binom{a-c-1}{u}\left(t_k^{-1} - 1\right)^u
	+ t^{I,-c,J}\sum_{u=1}^l\binom{a-c-1}{u-1}\left(t_k^{-1} - 1\right)^u \\
	&= t^{I,-c,J}\sum_{u=0}^l\binom{a-c-1}{u}\left(t_k^{-1} - 1\right)^u 
	+ t^{I,-c,J}\left(t_k^{-1} - 1\right)\sum_{u=0}^{l-1}\binom{a-c-1}{u}\left(t_k^{-1} - 1\right)^u \\
	&=t^{I,-c,J}\sum_{u=0}^l\binom{a-c-1}{u}\left(t_k^{-1} - 1\right)^u \\
	&\quad\quad + t^{I,-c,J}\left(t_k^{-1} - 1\right) 
		\left[ 
		\sum_{u=0}^{l}\binom{a-c-1}{u}\left(t_k^{-1} - 1\right)^u - \binom{a-c-1}{l}\left(t_k^{-1} - 1\right)^l
		\right] \\
	&=t^{I,-c-1,J}\sum_{u=0}^l\binom{a-c-1}{u}\left(t_k^{-1} - 1\right)^u - \binom{a-c-1}{l}t^{I,-c,J}\left(t_k^{-1} - 1\right)^{l+1}.
\end{align*}
The last subtracted term in the last line is further equal to $\binom{a-c-1}{l}t^{I,-c-l-1,J}\left(1-t_k\right)^{l+1}$ which is $\GG_\bfst$-equivalent to zero by Corollary \ref{cor:binom.G.relation}. This establishes the claim in \eqref{eqn:DD.proof.claim}. By inducting on $c$, we have that
\begin{multline*}
\GG_\bfst\left(
	\DD_k^{k+l}\left(t^{I,-a,J}\right) \right)  = \GG_\bfst\left(
		 t^{I,0,J}\sum_{u=0}^l\binom{a}{u}\left(t_k^{-1} - 1\right)^u
		 \right) \\
	 = \GG_\bfst\left(
		t^{I,-b,J}\sum_{u=0}^l\binom{a-b}{u}\left(t_k^{-1}-1\right)^u
		\right) \\
	 = \GG_\bfst\left(
		t^{I,-b,J}\left[1 + \left(t_k^{-1}-1\right)\right]^l
		\right) 
	 = \GG_\bfst\left(t^{I,-b-l,J}\right)
	\end{multline*}
where we used that $a-b=l$ in the third equality. Since $a=b+l$, this establishes the desired result for $b>0$. On the other hand, if $b=0$ (equivalently $a=l$) then each of the equalities above is immediate, and so the theorem is proved.
\end{proof}

\begin{remark}
When $l=0$ in the Theorem \ref{thm:DD.GG.Equiv}, the result should be viewed as a generalization to the content of Equation \eqref{eqn:GG.drop.nonpos}. That is, we observe that the operation $\DD_{k}^k$ is equivalent to setting the last $t$ variable (in this case $t_k$) equal to $1$ whenever the power on $t_k$ is non-positive. 
\end{remark}

\section{Extending the $\GG_\bfst$ operation to rational functions}
\label{s:GG.ratl}

%We need to prove that for the sequence $\mu$ with $\ell(\mu) = q$ and $N>\max(\mu)$, we have
%\[
%\GG_\bfst \left(
%t^\mu \, \sum_{(n_1,\ldots,n_q)\vdash N}
%	\left[ 
%		\prod_{j=1}^q \frac{(1-t_j)^{1-\delta_{0,n_j}}}{t_j^{n_j}}
%	\right]
%\right)=0.
%\]
%

\subsection{Residues and $\GG_\bfst$}
\label{ss:GG.and.RR}

We define the following operation on any rational function $f(z_1,\ldots,z_p)$,
\[
\RR_{(\bfs{z}|\bfs{\alpha};\bfs{\beta})}\left(f(z_1,\ldots,z_p)\right) := \\
  \IR_{z=0,\infty}\left(f(z_1,\ldots,z_p)\,\frac{\prod_{i=1}^p (1-z_i)^{- i}}{z_1\cdots z_p} \, P(\bfs{z}|\bfs{\alpha};\bfs\beta) \, \Delta(\bfs{z}) \,dz_p\cdots dz_1\right).
\]
Combining this with Equation \eqref{eqn:defn.G} and Definition \ref{defn:GG}, we immediately obtain the following theorem. 
\begin{thm}
\label{thm:GG.Laurent.poly}
For any Laurent polynomial $g(t_1,\ldots,t_p)$ we have
\[
\pushQED{\qed}
\GG_{\bfst}\left(g(t_1,\ldots,t_p)\right)(\bfs{\alpha};\bfs{\beta}) = \RR_{(\bfs{z}|\bfs{\alpha};\bfs{\beta})}\left(g(1-z_1,\ldots,1-z_p)\right).
\qedhere
\popQED
\]
\end{thm}
The goal of this section is extend the definition of $\GG_\bfst(g)$ to the case that $g$ is a rational function in such a way that Theorem \ref{thm:GG.Laurent.poly} remains true. We do not expect this to be possible for \emph{any} rational function $g$. However, an important special class of functions are those of the form
\begin{equation}
\label{eqn:f.I.J.defn}
f_{I,J}(\bfst) = 
t^{I,J}\, \prod_{i=1}^p \prod_{j=1}^q \frac{1-t_i}{1-t_i/t_{p+j}}
\end{equation}
for integer sequences $I\in\Z^p$ and $J\in\Z^q$. Our interest in $f_{I,J}$ is the following result.

\begin{prop}
\label{prop:GG.mult.RR}
For $I$, $J$, and $f_{I,J}(t_1,\ldots,t_{p+q})$ as above, we have
\begin{equation}
\label{eqn:GG.mult.RR}
G_I(\bfs\alpha;\bfs\beta)\,G_J(\bfs\alpha;\bfs\beta) = \RR_{(\bfs{z}|\bfs{\alpha};\bfs{\beta})}\left(f_{I,J}(1-z_1,\ldots,1-z_{p+q})\right)
\end{equation}
where $\bfs{z} = (z_1,\ldots,z_{p+q})$ and the result holds for any choice of $\bfs\alpha$ and $\bfs\beta$.
\end{prop}

\begin{proof}
Using Equation \eqref{eqn:defn.G}, the lefthand side of \eqref{eqn:GG.mult.RR} is
\begin{multline*}
\IR_{\bfs z'=0,\infty} \left(
	\left[ \frac{\prod_{i=1}^p (1-z'_i)^{I_i-i}}{z'_1 \cdots z'_p} P(\bfs z'|\bfs\alpha;\bfs\beta) \cdot \Delta(\bfs z')\,dz'_p \cdots dz'_1 \right] \right) \\
	\times 
\IR_{\bfs z''=0,\infty} \left(
	\left[ \frac{\prod_{j=1}^q (1-z''_j)^{J_j-j}}{z''_1 \cdots z''_q} P(\bfs z''|\bfs\alpha;\bfs\beta) \cdot \Delta(\bfs z'')\,dz''_q \cdots dz''_1 \right]
	\right).
\end{multline*}
After setting $z'_i=z_i$ and $z''_j = z_{p+j}$ the expression above is equal to
\begin{multline}
\label{eqn:RR1}
\IR_{\bfs z=0,\infty} \left(
	\left[ \frac{\prod_{i=1}^p (1-z_i)^{I_i-i}}{z_1 \cdots z_p} P(\bfs z'|\bfs\alpha;\bfs\beta) \cdot \Delta(\bfs z')\right] \right. \\
	\times 
\left.
	\left[ \frac{\prod_{j=1}^q (1-z_{p+j})^{J_j-j}}{z_{p+1} \cdots z_{p+q}} P(\bfs z''|\bfs\alpha;\bfs\beta) \cdot \Delta(\bfs z'') \right]\,dz_{p+q} \cdots dz_{p+1} \,dz_p \cdots dz_1 
	\right)
\end{multline}
where we used Fubini's theorem. Moreover, we observe that
\begin{gather*}
P(\bfs z|\bfs\alpha;\bfs\beta) = P(\bfs z'|\bfs\alpha;\bfs\beta) \cdot P(\bfs z''|\bfs\alpha;\bfs\beta),	\text{~and~} \\
\Delta(\bfs z) = \Delta(\bfs z') \cdot \Delta(\bfs z'') \cdot \prod_{i=1}^p\prod_{j=1}^q \left( 1-\frac{z_{p+j}}{z_i} \right)
\end{gather*}
and hence, the expression in \eqref{eqn:RR1} can be rewritten as
\begin{multline}
\label{eqn:GG.RR2}
\IR_{\bfs z=0,\infty} \left(
\left[
\frac{ \prod_{i=1}^p (1-z_i)^{I_i-i} \prod_{j=1}^q (1-z_{p+j})^{J_j-j-p} }
			{ z_1\cdots z_p \, z_{p+1} \cdots z_{p+q} }	
\right]
\right. \\
\times \left.
\left[
	\prod_{i=1}^p\prod_{j=1}^q 
		\frac{ (1-z_{p+j}) }
				{ \left( 1- \frac{z_{p+j}}{z_i} \right) }
	P(\bfs z | \bfs\alpha;\bfs\beta) \Delta(\bfs z)\right]
	\, dz_{p+q} \cdots dz_{p+1} \,dz_p \cdots dz_1 
	\right).
\end{multline}
Finally, under the transformation $t_r \mapsto 1-z_r$ observe that
\[
\frac{1-t_i}{1-t_i/t_{p+j}} \mapsto \frac{(1-z_{p+j})z_i}{z_i - z_{p+j}} = \frac{1-z_{p+j}}{1-\frac{z_{p+j}}{z_i}}.
\]
and so \eqref{eqn:GG.RR2} becomes
\[
\IR_{\bfs z=0,\infty} \left( 
	f_{I,J}(1-z_1,\ldots,1-z_{p+q}) \frac{\prod_{k=1}^{p+q} (1-z_k)^{-k}}{z_1\cdots z_{p+q}} 
		P(\bfs z|\bfs\alpha;\bfs\beta) \Delta(\bfs z) \, dz_{p+q} \cdots dz_1
	\right)
\]
which is $\RR_{(\bfs z|\bfs\alpha;\bfs\beta)} \left(f_{I,J}(1-z_1,\ldots,1-z_{p+q}) \right)$, as desired.
\end{proof}

\subsection{The case $q=1$}
\label{ss:q.1}

As above, we consider functions $f_{I,J}$ with $I\in \Z^p$ and $J\in\Z^q$. However, in this subsection we restrict to the case $q=1$. Begin by observing that the function $\frac{1-t_i}{1-t_i/t_{p+1}}$ has Laurent expansion 
\[
\sum_{u=0}^\infty (t_i^u/t_{p+1}^u) - \sum_{r=1}^\infty (t_i^r/t_{p+1}^{r-1}).
\]
Hence, we define the finite sum
\[
Q^N_i = { \sum_{u=0}^N  \frac{t_i^u}{t_{p+1}^u} - \sum_{r=1}^N \frac{t_i^r}{t_{p+1}^{r-1}} }.
\]
We need the following technical vanishing lemma.

\begin{lem} 
\label{lem:well.defined.GG}
Let $f\in \Z[t_1^{\pm 1},\ldots,t_{p+1}^{\pm 1}]$ be a Laurent polynomial and set
\begin{gather*}
g(t_1,\ldots,t_{p+1})  =f(t_1,\ldots,t_{p+1}) \prod_{i=1}^p \frac{1-t_i}{1-{t_i}/{t_{p+1}}},
\\
\widetilde{g}_N(t_1,\ldots,t_{p+1})  =f(t_1,\ldots,t_{p+1}) \prod_{i=1}^p Q_{i}^N.
\end{gather*}
Then for large enough $N$ and any choice of $\bfs\alpha$ and $\bfs\beta$, we have
\begin{equation}
\label{ww2}
\RR_{(\bfs{z}|\bfs\alpha;\bfs\beta)}\left(g (1-z_1,\ldots,1-z_{p+1})\right) =
\RR_{(\bfs{z}|\bfs\alpha;\bfs\beta)}\left(\widetilde{g}_N (1-z_1,\ldots,1-z_{p+1})\right).
\end{equation}
\end{lem}

\begin{proof}
Let $t_i=1-z_i$ for $i\in[p+1]$. The difference of the two sides of \eqref{ww2} is
\begin{equation}\label{ww3}
\RR_{(\bfs{z}|\bfs\alpha;\bfs\beta)}
\left(
f(1-z_1,\ldots,1-z_{p+1})
\left( \prod_{i=1}^p \frac{1-t_i}{1-t_i/t_{p+1}} -
\prod_{i=1}^p Q_{i}^N\right)
\right)
\end{equation}
Calculation shows that
\[
\frac{1-t_i}{1-t_i/t_{p+1}}-Q_{i}^N=\frac{1-t_i}{1-t_i/t_{p+1}}\cdot \frac{t_i^{N+1}}{t_{p+1}^{N+1}} \cdot \frac{1-t_{p+1}}{1-t_i}.
\]
Using this we obtain that
\begin{equation}
\label{ww1}
\prod_{i=1}^p \frac{1-t_i}{1-t_i/t_{p+1}} -
\prod_{i=1}^p Q_{i}^N  =
\prod_{i=1}^p \frac{1-z_{p+1}}{1-z_{p+1}/z_i}
\left(
\underbrace{
1-
\prod_{i=1}^p
\left(
1- \frac{(1-z_i)^{N+1}}{(1-z_{p+1})^{N+1}} \frac{z_{p+1}}{z_i}
\right)
}_{\spadesuit}
\right).
\end{equation}
The expression $\spadesuit$ expands to $2^p-1$ terms which can each be written as
\[\left( \frac{z_{p+1}}{ (1-z_{p+1})^{N+1}} \right)^r \psi(z_1,\ldots,z_p)\]
for some function $\psi$ and some integer $r\geq 1$. We obtain that \eqref{ww3} is the sum of $2^p-1$ terms, each of the form
\begin{equation}\label{ww4}
\RR_{(\bfs{z}|\bfs\alpha;\bfs\beta)}
\left( f \cdot
\left(\prod_{i=1}^p \frac{(1-z_{p+1})}{1-z_{p+1}/z_i}\right)
\left( \frac{z_{p+1}}{ (1-z_{p+1})^{N+1}} \right)^r
\psi(z_1,\ldots,z_p)
\right).
\end{equation}
Tracing back the definition of $\RR_{(\bfs{z}|\bfs\alpha;\bfs\beta)}$ and we find that \eqref{ww4} is equal to
\[
\IR_{z_1=0,\infty}\cdots\IR_{z_p=0,\infty}\left[\IR_{z_{p+1}=0,\infty}\left( f\cdot\frac{z_{p+1}^{r-1}\,P(z_{p+1}|\bfs\alpha;\bfs\beta)}{(1-z_{p+1})^{rN+r+1}}\,dz_{p+1}\right)\cdot\phi(z_1,\ldots,z_p)\,dz_p\cdots dz_1\right]
\]
for some function $\phi$ independent of $z_{p+1}$. Now by counting degrees of $z_{p+1}$, we observe that both $\Res_{z_{p+1}=0}$ and $\Res_{z_{p+1}=\infty}$ evaluate to zero for $N\gg0$. In the case that $f$ is an honest \emph{polynomial} in $t_{p+1}$, we remark that $N>\deg(f;t_{p+1})$ suffices.
\end{proof}

\begin{defn}
\label{defn:GG.ratl.defn}
For $f$, $g$, $\widetilde{g}_N$, and $N$ large enough such that Lemma \ref{lem:well.defined.GG} holds, we define
\[
\GG_\bfst\left(
f(t_1,\ldots,t_{p+1})\prod_{i=1}^p \frac{1-t_i}{1-t_i/t_{p+1}}
\right) =
\GG_\bfst\left(g(t_1,\ldots,t_{p+1})\right):=\GG_\bfst\left(\widetilde{g}_N(t_1,\ldots,t_{p+1})\right).
\]
\end{defn}

Now we further consider the case that $f = t^{L,j}$ with $L\in\Z^p$ and $j\leq 0$ in Lemma \ref{lem:well.defined.GG}. The proof of the lemma admits an obvious improvement to show that for any $1\leq k \leq p$ setting
\begin{gather*}
g(t_1,\ldots,t_{p+1}) = t^{L,j} \prod_{i=k}^p \frac{1-t_i}{1-t_i/t_{p+1}} \\
\widetilde{g}_N(t_1,\ldots,t_{p+1}) = t^{L,j} \prod_{i=k}^p Q_i^N
\end{gather*}
satisfies
\begin{equation}
\label{eqn:vanish.lemma.improve}
\RR_{(\bfs{z}|\bfs\alpha;\bfs\beta)}\left(
g(1-z_1,\ldots,1-z_{p+1})
\right)
=
\RR_{(\bfs{z}|\bfs\alpha;\bfs\beta)}\left(
\widetilde{g}_0(1-z_1,\ldots,1-z_{p+1})
\right).
\end{equation}
In other words, in this case $N=0$ suffices for the analogous result of Lemma \ref{lem:well.defined.GG} to hold. As a consequence, we obtain the following generalization of Equation \eqref{eqn:GG.drop.nonpos} to rational arguments of the $\GG_\bfst$ operator.

\begin{lem}
\label{lem:GG.eval.1.ratl}
For $L\in\Z^p$, $j\leq 0$, and $1\leq k \leq p$, we have
\begin{equation}
\label{eqn:GG.eval.1.ratl}
\GG_\bfst\left(
t^{L,j} \prod_{i=k}^p \frac{1-t_i}{1-t_i/t_{p+1}}
\right)
= \GG_\bfst\left(\left.
\left[ t^{L,j} \prod_{i=k}^p \frac{1-t_i}{1-t_i/t_{p+1}} \right]
\right|_{t_{p+1}=1}
\right).
\end{equation}
Therefore the result of both sides is $\GG_\bfst(t^L)$.
\end{lem}

\begin{proof}
Equation \eqref{eqn:vanish.lemma.improve} and Definition \ref{defn:GG.ratl.defn} imply that
\[
\GG_\bfst \left(
t^{L,j} \prod_{i=k}^p \frac{1-t_i}{1-t_i/t_{p+1}}
\right) = \GG_\bfst \left( t^{L,j} \right).
\]
Further, Equation \eqref{eqn:GG.drop.nonpos} implies $\GG_\bfst(t^{L,j}) = \GG_\bfst(t^L)$, and this is the result of the substitution $t_{p+1}\mapsto 1$ on the righthand side of \eqref{eqn:GG.eval.1.ratl}.
\end{proof}

As a further consequence of Lemma \ref{lem:well.defined.GG} we obtain the following multiplication formula which we will use in the sequel to investigate the $K$-Pieri rule.

\begin{thm}
\label{thm:GG.Pieri.mult}
For any partition $\lambda$ of length $p$ and any $n\in\Z_{>0}$ we have
	\begin{equation}
	\label{eqn:GG.Pieri.mult}
		G_\lambda\,G_{(n)} = \GG_{\bfst}
			\left(
			t^{\lambda,n}\prod_{i=1}^p \frac{1-t_i}{1-t_i/t_{p+1}}
			\right).
	\end{equation}
\end{thm}

\begin{proof}
Set $g$ to be the argument of the righthand side of \eqref{eqn:GG.Pieri.mult} and choose $N>n$. Lemma \ref{lem:well.defined.GG} holds for this $g$ and $N$, and so the righthand side is given by Definition \ref{defn:GG.ratl.defn}. We combine this with Equation \eqref{eqn:GG.mult.RR} using $I=\lambda$ and $J=(n)$ to see the desired product.
\end{proof}

\subsection{The case $q>1$}
\label{ss:q.greater.1}

Consider the rational function $f_{I,J}(t_1,\ldots,t_{p+q})$ of Equation \eqref{eqn:f.I.J.defn}. Recall that we wish to define $\GG_\bfst(f_{I,J})$ in such a way that the result of the operation is equal to the product $G_I\,G_J \in \Gamma$. In this subsection, we consider the case that $q=\ell(J)>1$. The main obstruction is the analogue of the vanishing result, Lemma \ref{lem:well.defined.GG}.

We propose the following approach. For any $k\in[p]$ consider the mapping 
\[
\curly{E}^d_k : \Z(t_1,\ldots,t_{p+q}) \longrightarrow \Z(t_1,\ldots,\widehat{t_k},\ldots,t_p,t_{p+1},\ldots,t_{p+q})[t_k^{\pm1}]
\]
which is the composition of sending $f_{I,J}$ to its expansion as a Laurent series about $t_k=0$, followed by truncating the series to have only degree less than or equal to $d$. Observe that every resulting term in the Laurent series will have degree exceeding $I_k$. Thus, truncating to terms with degree \emph{less than} $d$ indeed results in a Laurent \emph{polynomial} in $\Z(t_1,\ldots,\widehat{t_k},\ldots,t_p,t_{p+1},\ldots,t_{p+q})[t_k^{\pm1}]$.

For $d>I_k+N$, a computation shows that the coefficient of $t_k^{I_k + N}$ in $\curly{E}^{d}_k (f_{I,J})$ is
\begin{multline}
\label{eqn:coeff.t.k}
\widetilde{h}_k^N(t_1,\ldots,\widehat{t_k},\ldots,t_p,t_{p+1},\ldots,t_{p+q}) \\
= \left(t^{\hat{I},J} \prod_{i\in[p]\setminus\{k\}} \prod_{j\in q} \frac{1-t_i}{1-t_i/t_{p+j}}\right)
\cdot\sum_{(N_1,\ldots,N_q)\vdash N} \left(\prod_{j=1}^q \frac{(1-t_{p+j})^{1-\delta(N_j,0)}}{t_j^{N_j}} \right)
\end{multline}
where $\hat{I}$ is the integer sequence with $\hat{I}_i=I_i$ for all $i\neq k$ but $\hat{I}_k = 0$, $(N_1,\ldots,N_q)\vdash N$ denotes that $\sum_{j=1}^q N_j = N$ with $N_j\geq 0$, and $\delta$ is the Kronecker delta function. The following conjecture has been confirmed with many computer experiments.

\begin{conj}
\label{conj:RR.general.vanish}
Let $N>\max(J)$ and $\bfs{z}=(z_1,\ldots,z_{p+q})$. For every choice of $\bfs\alpha$ and $\bfs\beta$, we have
\[
\RR_{(\bfs{z}|\bfs\alpha;\bfs\beta)}\left( \widetilde{h}_k^N(1-z_1,\ldots,\widehat{1-z_k},\ldots,1-z_p,1-z_{p+1},\ldots,1-z_{p+q}) \right) = 0.
\]
\end{conj}

Iterating \eqref{eqn:coeff.t.k} leads to the realization that the composition $\curly{E}_1^{d_1}\curly{E}_2^{d_2}\cdots \curly{E}_p^{d_p}$ results in a Laurent \emph{polynomial} in $\Z[t_1^{\pm1},\ldots,t_{p+q}^{\pm1}]$ for any sequence of integers $(d_i)$; i.e.~the first product in \eqref{eqn:coeff.t.k} will be empty.

\begin{defn}
\label{defn:GG.ratl.q.greater.1}
Supposing the truth of Conjecture \ref{conj:RR.general.vanish} we define
\[
\GG_\bfst(f_{I,J}) := \GG_\bfst\left(\curly{E}_1^{d_1}\curly{E}_2^{d_2}\cdots \curly{E}_p^{d_p} (f_{I,J}) \right)
\]
provided that each $d_i$ exceeds $I_i+\max(J)$.
\end{defn}

\section{The operation $\SSS_\bfst$ on Laurent polynomials}
\label{s:SS.Lpoly}

We now make a short aside into the theory of Schur functions. To each partition $\lambda$, there is an associated Schur function $s_\lambda$. The Schur functions form a basis of the ring of symmetric functions $\Lambda$. Moreover, they satisfy the Jacobi--Trudi formula
\begin{equation}
\label{eqn:Jacobi.Trudi}
s_\lambda = \det\left( h_{\lambda_i + j - i} \right)
\end{equation} 
where $h_d$ denotes the \emph{complete homogeneous symmetric function} of degree $d$ \cite[Ch.\ I, (3.4)]{im1995}. The Schur functions can be evaluated on a single set of variables $\bfs{x} = \{x_1,x_2,\ldots\}$ for which we write the Schur function $s_\lambda(\bfs x)$. In this case, the polynomials $h_d$ are formally defined by the generating function
\[
\sum_{d\geq 0} h_d(\bfs x) u^d = \frac{1}{\prod_{i\geq 1} (1- x_i u)}.
\]
There is also a \emph{supersymmetric Schur function} $s_\lambda(\bfs x;\bfs y)$ separately symmetric in the families $\bfs x$ and $\bfs y = \{y_1,y_2,\ldots\}$. In this case, the polynomials $h_d(\bfs x;\bfs y)$ are defined by the generating function
\[
\sum_{d\geq 0} h_d(\bfs x;\bfs y) u^d = \frac{\prod_{j\geq 1} (1 + y_j u)}{\prod_{i\geq 1} (1 - x_i u) }.
\]

\begin{remark}
For any partition $\lambda$, the lowest degree homogeneous part of $G_\lambda(\bfs\alpha;\bfs\beta)$ (of degree $|\lambda|$) is $s_\lambda(\bfs x;\bfs y)$ after making the substitutions $x_i = 1-\alpha_i^{-1}$ and $y_j = 1-\beta_j$, see e.g.\ \cite{ab2002.klr}. In the sequel, we will abuse notation by writing $G_\lambda(\bfs{x};\bfs{y})$ for the polynomial $G_\lambda(\{1/(1-x_i)\};\{1-y_j\})$.
\end{remark}

The Jacobi--Trudi determinant \eqref{eqn:Jacobi.Trudi} can be adapted to define a Schur function $s_I$ for any finite integer sequence $I$ by
\[
s_I := \det\left( h_{I_i + j - i} \right).
\]
Thus, by interchanging rows in the determinant, we obtain the straightening law, valid for any finite integer sequences $I$ and $J$ and for any integers $a$ and $b$,
\[
s_{I,a,b,J} = -s_{I,b-1,a+1,J}.
\]
Thus, $s_I = \pm s_\nu$ for some partition $\nu$, or is zero. Observe that the Schur straightening law can also be viewed as a consequence of that for Grothendieck polynomials by taking the lowest degree part on each side of Equation \eqref{eqn:G.Straighten}.

\begin{defn}
\label{defn:CurlyS}
For any monomial $t^I$ we define
\[
\SSS_\bfst\left( t^I \right) = s_I
\]
and extend the operation linearly to obtain a $\Z$-module mapping $\SSS_\bfst:\Z[\bfst^{\pm1}] \to \Lambda$. %If two Laurent polynomials $f_1$ and $f_2$ have the property that $\SSS_\bfst\left(f_1-f_2\right) = 0 \in \Lambda$, then we say that $f_1$ is \emph{$\SSS_\bfst$-equivalent} to $f_2$.
\end{defn}

\begin{remark}
\label{rem:low.deg.of.GG}
The observation regarding lowest degree terms of $G_\lambda$ for partitions $\lambda$ generalizes to $G_I$ with $I$ any finite integer sequence as follows. First write $G_I = \sum_\lambda d^I_\lambda G_\lambda$ using the straightening laws. Observe that as a consequence of the straightening laws, for any $J$ we have $d^J_\lambda \neq 0$ only if $|\lambda|\geq |J|$. If it happens that $d_\lambda^I \neq 0$ only if $|\lambda|$ is strictly more than $|I|$, then $\SSS_\bfst(t^I) = 0$. Otherwise the lowest degree homogeneous term of $G_I$ is $s_I$. More succinctly, the degree $|I|$ part of $G_I(\bfs x;\bfs y)$ is always $s_I(\bfs x;\bfs y)$, which may be zero, although $G_I(\bfs x;\bfs y)$ is never itself zero.
\end{remark}
 
\section{Pieri rules}
\label{s:K.Pieri.rule}

We wish to investigate the multiplication $G_\lambda \cdot G_{(n)}$ completely in terms of iterated residue operations. In particular, we will explore interpretations of the following Theorems \ref{thm:K.Pieri.rule} and \ref{thm:H.Pieri.rule}.

\begin{thm}[Lenart \cite{cl2000}]
\label{thm:K.Pieri.rule}
There exist unique \underline{positive} integers $c_{\lambda,n}^\mu$ such that
\[
G_{\lambda}\cdot G_{(n)} = \sum_\mu (-1)^{|\mu|-|\lambda|-n}\,c_{\lambda,n}^\mu G_\mu
\]
where the sum is over partitions $\mu$ satisfying certain combinatorial conditions. Furthermore, only \underline{finitely} many of the integers $c_{\lambda,n}^\mu$ are nonzero.
\end{thm}

The above result is called the \emph{$K$-Pieri rule}. We choose to emphasize the \emph{positivity} and \emph{finiteness} aspects of Lenart's original paper \cite{cl2000}. We remark that in that work, a complete combinatorial description of the coefficients is given. In short, the number $c_{\lambda,n}^\mu$ is nonzero only if $\mu$ can be obtained from $\lambda$ by adding a horizontal strip and in that case is equal to an explicitly described binomial coefficient. For details, see \cite[Theorem~3.2]{cl2000}.

We will also give a combinatorial translation for our iterated residue results, Theorem \ref{thm:GG.K.Pieri.comb}, comparable to a combinatorial formulation of the \emph{cohomological} Pieri rule, stated below (see e.g.~\cite[Ch.\ I, (5.16)]{im1995}).

%\begin{defn}[see e.g.\ \cite{im1995}, Ch.\ I, $\S$1]
%\label{defn:hor.n.strip}
%Given a partition $\mu\supset\lambda$, the \emph{skew diagram} $\mu/\lambda$ is the Young diagram of boxes formed by removing the shape $\lambda$ from the northwest corner of the shape $\mu$. The skew diagram $\mu/\lambda$ is called a \emph{horizontal $n$-strip} if its weight $|\mu/\lambda|:=|\mu|-|\lambda| = n$ and no two boxes of $\mu/\lambda$ are in the same column.
%\end{defn}

\begin{thm}(Cohomological Pieri rule)
\label{thm:H.Pieri.rule}
For $\lambda$ and $n$ as above,
\[
s_\lambda\cdot s_{(n)} = \sum_{\mu} s_\mu
\]
where the sum is over partitions $\mu$ such that $\mu/\lambda$ is a horizontal $n$-strip.
\end{thm}

We will prove Theorem \ref{thm:K.Pieri.rule} in Sections \ref{s:K.Pieri.finiteness} and \ref{s:K.Pieri.positivity} using iterated residue methods and interpret our results in the combinatorial context of Theorem \ref{thm:H.Pieri.rule} in Section \ref{s:translate.comb}.

\section{An iterated residue proof of finiteness in the $K$-Pieri rule}
\label{s:K.Pieri.finiteness}

Throughout the rest of the paper, fix positive integers $p$ and $n$; moreover, fix a partition $\lambda$ of length $p$. Let $\Comp_{p+1}(n)$ denote the set of weak compositions of $n$ into $p+1$ parts; that is, 
\[ 
	\Comp_{p+1}(n) = 
		\{ I \in \N^{p+1} : |I|=n \}.
\]
Observe that the set $\Comp_{p+1}(n)$ can be stratified by lengths of its elements. Indeed, we set $\Comp_{p+1}^k(n) = \{ I \in \Comp_{p+1}(n) : \ell(I) = k \}$ so that $\Comp_{p+1}(n) = \DisjointUnion_{k\in[p+1]} \Comp_{p+1}^k(n)$.

\begin{defn}
\label{defn:T.polys}
Given $I\in\Comp_{p+1}^k(n)$ and a vector $\epsilon \in \{0,1\}^{k-1}$ define the polynomial
\begin{equation}
\label{eqn:A.lambda.I.epsilon}
\T{\lambda}{I}{\epsilon} := t^{\lambda+I}\,\prod_{i=1}^{k-1} (1-t_i)^{\epsilon_i}.
\end{equation}
When $\epsilon = (1,\ldots,1)$ we will simply write $\T{\lambda}{I}{}$ as a shorthand. 
\end{defn}

Already, Definition \ref{defn:GG.ratl.defn} and Theorem \ref{thm:GG.Pieri.mult} prove the \emph{finiteness} assertion of Theorem \ref{thm:K.Pieri.rule}. However, the following result explicitly identifies the finitely many terms we need to keep from the expansion of $\widetilde{g}_{n+1} = t^{\lambda,n} \prod_{i=1}^{p} Q_i^{n+1}$.

\begin{thm}
\label{thm:GG.finiteness}
For a partition $\lambda$ of length $p$ and positive integer $n$, we have
\[
G_\lambda\cdot G_{(n)} = \GG_\bfst \left( t^{\lambda,n} \, \prod_{i=1}^p \frac{1-t_i}{1-t_i/t_{p+1}}\right)
	= \GG_\bfst \left(\sum_{I\in\Comp_{p+1}(n)} \T{\lambda}{I}{} \right).
\]
\end{thm}

\begin{proof}
We will expand the terms $1/(1-t_i/t_{p+1})$ in the rational function above in  order of increasing $i$ and repeatedly apply Equation \eqref{eqn:GG.eval.1.ratl} whenever the exponents on $t_{p+1}$ become non-positive. Set
\[
g_{r} = \prod_{i=r}^p \frac{1-t_i}{1-t_i/t_{p+1}}.
\]
First, we expand
\[
t^{\lambda,n}\,g_1
	= \mathop{ \underbrace{ \mathop{\left( \sum_{u=0}^{n-1} t_1^u \cdot t^{\lambda,n} \cdot t_{p+1}^{-u} \right)}}}_{f_1}(1-t_1)\, g_2 + t_1^{n} \cdot t^{\lambda,n} \cdot t_{p+1}^{-n} \cdot g_1.
\]
The last term above, once we apply \eqref{eqn:GG.eval.1.ratl} is $\GG_\bfst$-equivalent to $t_1^n t^{\lambda} = \T{\lambda}{(n,0,\ldots,0)}{}$ and corresponds to the unique element of $\Comp_{p+1}^1(n)$. Now choose one of the terms in the  parenthetical summation $f_1$, say corresponding to $u=I_1$ for some $I_1\in\{0,\ldots,n-1\}$. We expand
\[
t_1^{I_1} t^{\lambda,n} t_{p+1}^{-I_1}\,(1-t_1)\, g_2 = 
	\mathop{\underbrace{\mathop{\left( \sum_{u=0}^{n-I_1-1} t_1^{I_1}t_2^u t^{\lambda,n} t_{p+1}^{-I_1-u} \right)}}}_{f_2}
	(1-t_1) (1-t_2)\, g_3 + t_1^{I_1}t_2^{n-I_1} t^{\lambda,n} t_{p+1}^{-n}(1-t_1)g_2
\]
where the last term, again applying \eqref{eqn:GG.eval.1.ratl}, is $\GG_\bfst$-equivalent to $t_1^{I_1}t_2^{n-I_1} t^{\lambda}(1-t_1)$ which is $\T{\lambda}{I}{}$ with $I=(I_1,I_2, 0,\ldots,0)$ and $I_2 = n-I_1$. In this way we can obtain each $\T{\lambda}{I}{}$ with $I\in\Comp_{p+1}^2(n)$, since $I_1<n$. 

Similarly applying the analogous procedure to each term of the summation $f_2$, we obtain the polynomials $\T{\lambda}{I}{}$ with $I\in\Comp_{p+1}^3(n)$. We can continue this procedure until we have expansions of the form
\begin{multline*}
\prod_{i=1}^{p-1}t_i^{I_i} \cdot t^{\lambda,n} \cdot t_{p+1}^{-\sum_{i=1}^{p-1} I_i} \cdot \prod_{i=1}^{p-1}(1-t_i) \cdot g_p= \\
	\mathop{\underbrace{\mathop{\left( \sum_{u=0}^{n-(\sum_{i=1}^{p-1} I_i) - 1} \prod_{i=1}^{p-1}t_i^{I_i} \cdot t_p^u\cdot t^{\lambda,n} \cdot t_{p+1}^{-(\sum_{i=1}^{p-1} I_i) - u} \right)}}}_{f_{p}} \prod_{i=1}^{p}(1-t_i) \\
	+ \prod_{i=1}^{p-1} t_i^{I_i} \cdot t_p^{n-(\sum_{i=1}^{p-1} I_i)} \cdot t^{\lambda,n} \cdot t_{p+1}^{-n}\cdot \prod_{i=1}^{p-1}(1-t_i)\cdot g_p.
\end{multline*}
As before, we can apply \eqref{eqn:GG.eval.1.ratl} to the last displayed line to obtain the polynomial $\T{\lambda}{I}{}$ for $I = (I_1,\ldots,I_p,0) \in \Comp^p_{p+1}(n)$. Finally, each of the terms in the summation $f_{p}$ is already a polynomial, and when multiplied by the common factor $\prod_{i=1}^p(1-t_i)$, is of the form $\T{\lambda}{I}{}$ with $I=(I_1,\ldots,I_p,I_{p+1}) \in \Comp_{p+1}^{p+1}(n)$. Explicitly we take $u=I_p$ and thus $I_{p+1} = n-\sum_{i=1}^{p} I_i = n - \sum_{i=1}^{p-1} I_i -u > 0$, so indeed $\ell(I)=p+1$.

We observe that there are no more rational functions to expand as Laurent series, and each polynomial term we have produced has the form $\T{\lambda}{I}{}$. Conversely, it is clear from our description that for any $I\in\Comp_{p+1}(n)$, we can follow this procedure to obtain $\T{\lambda}{I}{}$ by successively taking terms (using $u$ as a dummy variable as above) corresponding to $u=I_1$ in $f_1$, $u=I_2$ in $f_2$, \emph{et cetera}, until we take $u=I_{p}$ in $f_p$.
\end{proof}

\section{An iterated residue proof of positivity in the $K$-Pieri rule}
\label{s:K.Pieri.positivity}

\subsection{Preliminary results and definitions}
\label{ss:K.Pieri.positivity.prelim}

In this subsection we present the technical definitions, lemmas, and propositions which allow us to organize a cancellation of the non-sorted terms on the righthand side of Theorem \ref{thm:GG.finiteness}. 

\begin{defn}
\label{defn:sorted.polys}
Given a finite integer sequence $I$ with only non-negative parts and positive integer $m$, the monomial $t^I$ is called \emph{$m$-sorted} if $I_m\geq I_{m+1}$. Recall that if $k>\ell(I)$ then $I_k=0$, so $t^I$ is trivially $m$-sorted for all $m\geq\ell(I)$. We say that $t^I$ is \emph{sorted} if it is $m$-sorted for all positive integers $m$. Observe that $t^I$ is sorted if and only if $I$ is a partition. A polynomial in the variables $t$ is $m$-sorted (respectively sorted) if every monomial in its expansion is $m$-sorted (resp.~sorted).
\end{defn}

The following Proposition \ref{prop:when.sorted} makes the observation of exactly when the polynomials $\T{\lambda}{I}{\epsilon}$ are $m$-sorted. Furthermore, Lemma \ref{lem:m.bound} establishes that $\T{\lambda}{I}{\epsilon}$ can fail to be $m$-sorted for only finitely many $m$.

\begin{prop}
\label{prop:when.sorted}
When $\epsilon_{m+1} = 0$, the polynomial $\T{\lambda}{I}{\epsilon}$ is $m$-sorted if and only if $\lambda_m+I_m \geq \lambda_{m+1}+I_{m+1}$. When $\epsilon_{m+1} = 1$, the polynomial $\T{\lambda}{I}{\epsilon}$ is $m$-sorted if and only if $\lambda_m + I_m > \lambda_{m+1}+I_{m+1}$.
\end{prop}

\begin{proof}
When $\epsilon_{m+1} = 0$, the exponent of $t_{m+1}$ in every monomial term of $\T{\lambda}{I}{\epsilon}$ is $\lambda_{m+1}+I_{m+1}$. In a given term, the exponent on $t_{m}$ is $\lambda_m+I_m + \epsilon_m \geq \lambda_m+I_m$. Hence, for either choice of $\epsilon_m \in \{0,1\}$ a monomial will be $m$-sorted exactly when $\lambda_m+I_m \geq \lambda_{m+1}+I_{m+1}$.

On the other hand, for $\epsilon_{m+1} = 1$, terms with $E_m = \lambda_m + I_m$ as the exponent on $t_m$ and $E_{m+1} = \lambda_{m+1}+I_{m+1} + 1$ as the exponent on $t_{m+1}$ will appear in the expansion of $\T{\lambda}{I}{\epsilon}$. Hence, the strict inequality is necessary for the claim to hold since $\lambda_m + I_m = \lambda_{m+1}+I_{m+1}$ contradicts $E_m \geq E_{m+1}$.
\end{proof}

\begin{lem}
\label{lem:m.bound}
If $\T{\lambda}{I}{\epsilon}$ is not $m$-sorted, then $m<\ell(I)$.
\end{lem}

\begin{proof}
Suppose $m\geq \ell(I)$ and $\T{\lambda}{I}{\epsilon}$ is not $m$-sorted. By definition we have that $\epsilon_{m} = \epsilon_{m+1} = 0$ and $I_{m+1} = 0$. Thus every monomial term in the expansion of $\T{\lambda}{I}{\epsilon}$ has the exponents $E_m = \lambda_m + I_m$ and $E_{m+1} = \lambda_{m+1}$ respectively on the $t_m$ and $t_{m+1}$ variables. If this monomial is not $m$-sorted, then $E_m<E_{m+1}$, which is impossible since $\lambda$ is a partition (i.e.~$\lambda_m\geq\lambda_{m+1}$) and $I_m\geq 0$.
\end{proof}

The following definition characterizes the integer sequences $I\in\Comp_{p+1}(n)$ for which $\T{\lambda}{I}{}$ is as far from being $m$-sorted as possible.

\begin{defn}
\label{defn:Bad}
Define the set
\begin{equation}
\label{eqn:Bad}
\Bad_{k,m} = \{ I \in \Comp_{p+1}^k(n) : I_m = 0 \text{~and $\T{\lambda}{I}{}$ is not $m$-sorted} \}.
\end{equation}
\end{defn}

\begin{remark}
\label{rem:Bad.empty}
Lemma \ref{lem:m.bound} ensures that $\Bad_{k,m}$ is empty whenever $m\geq k$. Depending on $\lambda$, the set $\Bad_{k,m}$ may or may not be empty for other values of the parameters. However, observe that if $\T{\lambda}{I}{\epsilon}$ is not $m$-sorted and $I_m\neq 0$, then for any integer sequence $I'$ with $I'_i = I_i$ for $i\neq m,m+1$ and $I'_m<I_m$, the polynomial $\T{\lambda}{I'}{\epsilon}$ is also not $m$-sorted. Thus the condition $I_m=0$ in $\Bad_{k,m}$ indeed forces its elements to be as ``bad'' as possible in terms of how far the corresponding $\T{\lambda}{I}{}$ are from being $m$-sorted. On the other hand, this observation implies that if there exists a length $k$ sequence $J$ such that $\T{\lambda}{J}{}$ is not $m$-sorted, then the set $\Bad_{k,m}$ is non-empty.
\end{remark}

\begin{defn}
\label{defn:Seg.Max}
For $I\in \Bad_{k,m}$, define the set
\begin{equation}
\label{eqn:Seg}
\Seg_m(I) = \left\{ J \in \Comp_{p+1}^k(n) :
	\begin{array}{c} 
		J_i = I_i \text{~for all~} i\neq m,m+1 \\
		\text{and} \\
	 	\lambda_m - \lambda_{m+1} \leq J_{m+1} \leq I_{m+1}
	\end{array}
	\right\}.
\end{equation}
When additionally $k>m+1$ or $\lambda_m > \lambda_{m+1}$ define the sequence
\begin{equation}
\label{eqn:Max}
\Max_m(I) = \\ (I_1,\ldots,I_{m-1},
	I_{m+1} - (\lambda_m-\lambda_{m+1}) , \lambda_m - \lambda_{m+1},
	 I_{m+2},\ldots).
\end{equation}
\end{defn}

We note that the condition $I\in\Bad_{k,m}$ guarantees that $I_{m+1} -(\lambda_m - \lambda_{m+1})\geq 0$, so that $\Max_m(I)$ has only non-negative parts. Moreover, notice that $|\Max_m(I)| = |I|$ and hence $\Max_m(I) \in \Comp_{p+1}(n)$. In fact, $\Max_m(I)\in\Seg_m(I)$; note the importance of the condition $k>m+1$ or $\lambda_m>\lambda_{m+1}$ to make this true. Moreover, $\Max_m(I)$ is indeed a ``maximal'' element of $\Seg_m(I)$ in the sense that its $m$-th entry is as large as possible given the other constraints. The effect of doing $\lambda+J$ as $J$ ranges through $\Seg_m(I)$ is depicted in Figures \ref{fig:Seg.generic} and \ref{fig:Seg.special}.

\begin{figure}
\begin{tabular}{ccccc}
\begin{tikzpicture}[scale=0.35]
\draw (0,2) -- (0,0) -- (1,0) -- (1,1) -- (2,1) -- (2,2);
\draw (0,2) rectangle (2,3);
\draw (0,3) rectangle (3,4);
\draw (0,4) -- (0,6) -- (4,6) -- (4,5) -- (3,5) -- (3,4);
\node (lambdam) at (1.5,3.5) {{\tiny$\lambda_m$}};
\node (lambdam+1) at (1,2.5) {{\tiny$\lambda_{m+1}$}};

\filldraw[fill=cyan!30, draw=black] (4,6) -- (6,6) -- (6,5) -- (7,5) -- (7,4) -- (3,4) -- (3,5) -- (4,5) -- cycle;
\filldraw[fill=cyan!30, draw=black] (2,1) rectangle (4,2);

\filldraw[fill=red!30, draw=black] (3,3) rectangle (3,4);
\filldraw[fill=red!30, draw=black] (2,2) rectangle (7,3);
%\node (Jm) at (??,3.5) {\tiny$?$};
\node (Jm+1) at (4.5,2.5) {\tiny$I_{m+1}$};
\end{tikzpicture}
&,
&
\begin{tikzpicture}[scale=0.35]
\draw (0,2) -- (0,0) -- (1,0) -- (1,1) -- (2,1) -- (2,2);
\draw (0,2) rectangle (2,3);
\draw (0,3) rectangle (3,4);
\draw (0,4) -- (0,6) -- (4,6) -- (4,5) -- (3,5) -- (3,4);
\node (lambdam) at (1.5,3.5) {{\tiny$\lambda_m$}};
\node (lambdam+1) at (1,2.5) {{\tiny$\lambda_{m+1}$}};

\filldraw[fill=cyan!30, draw=black] (4,6) -- (6,6) -- (6,5) -- (7,5) -- (7,4) -- (3,4) -- (3,5) -- (4,5) -- cycle;
\filldraw[fill=cyan!30, draw=black] (2,1) rectangle (4,2);

\filldraw[fill=red!30, draw=black] (3,3) rectangle (4,4);
\filldraw[fill=red!30, draw=black] (2,2) rectangle (6,3);
\node (Jm) at (3.5,3.5) {\tiny$1$};
\node (Jm+1) at (4,2.5) {\tiny$I_{m+1}-1$};
\end{tikzpicture}
&$,\cdots,$
&
\begin{tikzpicture}[scale=0.35]
\draw (0,2) -- (0,0) -- (1,0) -- (1,1) -- (2,1) -- (2,2);
\draw (0,2) rectangle (2,3);
\draw (0,3) rectangle (3,4);
\draw (0,4) -- (0,6) -- (4,6) -- (4,5) -- (3,5) -- (3,4);
\node (lambdam) at (1.5,3.5) {{\tiny$\lambda_m$}};
\node (lambdam+1) at (1,2.5) {{\tiny$\lambda_{m+1}$}};

\filldraw[fill=cyan!30, draw=black] (4,6) -- (6,6) -- (6,5) -- (7,5) -- (7,4) -- (3,4) -- (3,5) -- (4,5) -- cycle;
\filldraw[fill=cyan!30, draw=black] (2,1) rectangle (4,2);

\filldraw[fill=red!30, draw=black] (3,3) rectangle (7,4);
\filldraw[fill=red!30, draw=black] (2,2) rectangle (3,3);
\node (Jm) at (10,5.5) {\tiny$I_{m+1} - (\lambda_m - \lambda_{m+1})$};
\node (Jm+1) at (7,1.5) {\tiny$\lambda_m - \lambda_{m+1}$};

\draw (Jm) -- (5,3.5);
\draw (Jm+1) -- (2.5,2.5);
\end{tikzpicture}
\end{tabular}
\caption{Diagrams of $\lambda + J$ for $J\in\Seg_m(I)$ with $I\in\Bad_{k,m}$. In this example, $k>m+1$ and $\lambda_m - \lambda_{m+1}>0$ so the diagram on the far left represents $\lambda + I$ while the diagram on the far right is $\lambda + \Max_m(I)$. The Young diagram for $\lambda$ is unshaded. The blue shading represents the addition of $I_i$ for $i\neq m,m+1$, and the red shading represents the different values of $J_m$ and $J_{m+1}$ as $J$ ranges through $\Seg_m(I)$.}
\label{fig:Seg.generic}
\end{figure}
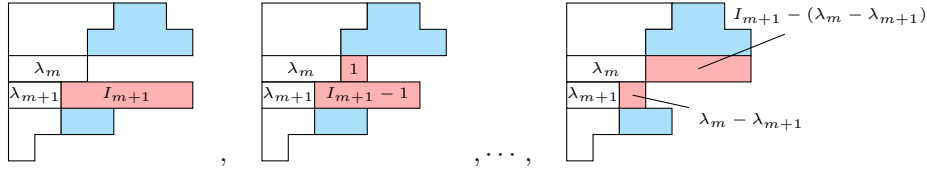

%%%%%%%%%%%%%%%%%%%%%%%%%%%%%%%%%%%%%%%%%%%%%%%%%%%%%%%%%%%%%%%%%%%%%%%%%%%%%%%%%%%%%%

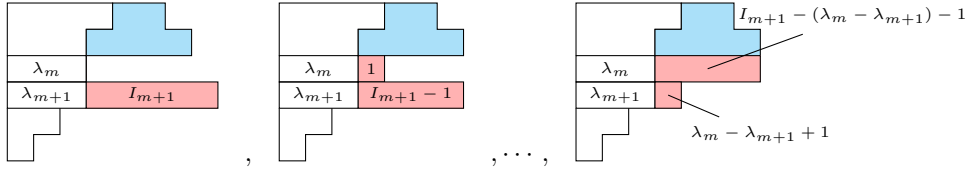
\begin{figure}
\begin{tabular}{ccccc}
\begin{tikzpicture}[scale=0.35]
\draw (0,2) -- (0,0) -- (1,0) -- (1,1) -- (2,1) -- (2,2);
\draw (0,2) rectangle (3,3);%\lambda_{m+1}
\draw (0,3) rectangle (3,4);%\lambda_m
\draw (0,4) -- (0,6) -- (4,6) -- (4,5) -- (3,5) -- (3,4);
\node (lambdam) at (1.5,3.5) {{\tiny$\lambda_m$}};
\node (lambdam+1) at (1.5,2.5) {{\tiny$\lambda_{m+1}$}};

\filldraw[fill=cyan!30, draw=black] (4,6) -- (6,6) -- (6,5) -- (7,5) -- (7,4) -- (3,4) -- (3,5) -- (4,5) -- cycle;

\filldraw[fill=red!30, draw=black] (3,3) rectangle (3,4);
\filldraw[fill=red!30, draw=black] (3,2) rectangle (8,3);
%\node (Jm) at (??,3.5) {\tiny$?$};
\node (Jm+1) at (5.5,2.5) {\tiny$I_{m+1}$};
\end{tikzpicture}
&,
&
\begin{tikzpicture}[scale=0.35]
\draw (0,2) -- (0,0) -- (1,0) -- (1,1) -- (2,1) -- (2,2);
\draw (0,2) rectangle (3,3);%\lambda_{m+1}
\draw (0,3) rectangle (3,4);%\lambda_m
\draw (0,4) -- (0,6) -- (4,6) -- (4,5) -- (3,5) -- (3,4);
\node (lambdam) at (1.5,3.5) {{\tiny$\lambda_m$}};
\node (lambdam+1) at (1.5,2.5) {{\tiny$\lambda_{m+1}$}};

\filldraw[fill=cyan!30, draw=black] (4,6) -- (6,6) -- (6,5) -- (7,5) -- (7,4) -- (3,4) -- (3,5) -- (4,5) -- cycle;

\filldraw[fill=red!30, draw=black] (3,3) rectangle (4,4);
\filldraw[fill=red!30, draw=black] (3,2) rectangle (7,3);
\node (Jm) at (3.5,3.5) {\tiny$1$};
\node (Jm+1) at (5,2.5) {\tiny$I_{m+1}-1$};
\end{tikzpicture}
&$,\cdots,$
&
\begin{tikzpicture}[scale=0.35]
\draw (0,2) -- (0,0) -- (1,0) -- (1,1) -- (2,1) -- (2,2);
\draw (0,2) rectangle (3,3);%\lambda_{m+1}
\draw (0,3) rectangle (3,4);%\lambda_m
\draw (0,4) -- (0,6) -- (4,6) -- (4,5) -- (3,5) -- (3,4);
\node (lambdam) at (1.5,3.5) {{\tiny$\lambda_m$}};
\node (lambdam+1) at (1.5,2.5) {{\tiny$\lambda_{m+1}$}};

\filldraw[fill=cyan!30, draw=black] (4,6) -- (6,6) -- (6,5) -- (7,5) -- (7,4) -- (3,4) -- (3,5) -- (4,5) -- cycle;

\filldraw[fill=red!30, draw=black] (3,3) rectangle (7,4);
\filldraw[fill=red!30, draw=black] (3,2) rectangle (4,3);
\node (Jm) at (10.5,5.5) {\tiny$I_{m+1} - (\lambda_m - \lambda_{m+1})-1$};
\node (Jm+1) at (7,1) {\tiny$\lambda_m - \lambda_{m+1}+1$};

\draw (Jm) -- (5,3.5);
\draw (Jm+1) -- (3.5,2.5);
\end{tikzpicture}

\end{tabular}
\caption{Diagrams of $\lambda + J$ for $J\in\Seg_m(I)$ with $I\in\Bad_{k,m}$, $k=m+1$, and $\lambda_m - \lambda_{m+1}=0$ so the diagram on the far left still represents $\lambda + I$, but $\Max_m(I)$ is not defined (it would have length less than $I$). However, setting $\Max_m(I)$ to be the formula in Equation \eqref{eqn:Max}, the far right diagram depicts $\lambda+J$ for $J = \Max_m(I) + (0,\ldots,0,-1,1,0,\ldots,0)$, where the $-1$ and $1$ appear in the $m$-th and $(m+1)$-st components of the added vector.}
\label{fig:Seg.special}
\end{figure}

\subsection{Organization of cancellations among the polynomials $\T{\lambda}{I}{\epsilon}$}
\label{ss:K.Pieri.positivity.organize.cancellation}

\begin{defn}
\label{defn:cancel.segments}
Let $a\leq b$ be integers and let $i$ be a positive integer. Set
\begin{equation}
\label{eqn:sigma}
\sigma_i(a,b) = \sum_{k=0}^{b-a} t_i^{a+k}t_{i+1}^{b-k} = t_i^{a}t_{i+1}^{b} + t_i^{a+1}t_{i+1}^{b-1} + \cdots + t_i^{b}t_{i+1}^{a}.
\end{equation}
If $a$ is strictly less than $b$, then we also define
\begin{equation}
\label{eqn:sigma.tilde}
\tilde{\sigma}_i(a,b) = \sum_{k=0}^{b-a-1} t_i^{a+k}t_{i+1}^{b-k} = t_i^{a}t_{i+1}^{b} + t_i^{a+1}t_{i+1}^{b-1} + \cdots + t_i^{b-1}t_{i+1}^{a+1}.
\end{equation}
A \emph{cancelling segment (of Type A, B, or C) on $t_i$ and $t_{i+1}$} is a polynomial with one of the following forms
	\begin{enumerate}[leftmargin=*,label=\underline{Type \Alph*.}]
	\item $f(t)\,\tilde{\sigma}_i(a,b)\,(1-t_i)$
	\item $f(t)\,\sigma_i(a,b)\,(1-t_i)$
	\item $f(t)\,\sigma_i(a,b)\,(1-t_i)(1-t_{i+1})$
	\end{enumerate}
where $f(t)$ is allowed to be any polynomial which is independent of $t_i$ and $t_{i+1}$.
\end{defn}

The importance of cancelling segments is the following proposition.

\begin{prop}
\label{prop:cancel.segments}
In the notation of Definition \ref{defn:cancel.segments}:
	\begin{enumerate}[leftmargin=*,label=(\Alph*)]
	\item $f(t)\,\tilde{\sigma}_i(a,b)\,(1-t_i)$ is $\GG_\bfst$-equivalent to $0$;
	\item $f(t)\,\sigma_i(a,b)\,(1-t_i)$ is $\GG_\bfst$-equivalent to $f(t)\,t_i^b t_{i+1}^a (1-t_i)$;
	\item $f(t)\,\sigma_i(a,b)\,(1-t_i)(1-t_{i+1})$ is also $\GG_\bfst$-equivalent to $f(t)\, t_i^b t_{i+1}^a (1-t_i)$.
	\end{enumerate}
\end{prop}

\begin{proof}
We note that (A) immediately implies (B) and moreover, the expression in (C) can be rewritten as
\[
f\,\sigma_i(a,b)\,(1-t_i)(1-t_{i+1}) = f\,\sigma_i(a,b)(1-t_i) - f\,\tilde{\sigma}_i(a,b+1)(1-t_i)
\]
and so using (B) on the first term and (A) on the second term yields the result. Hence, we need to prove (A). To simplify notation, it suffices to consider the case when $f=1$ and $i=1$. The straightening law \eqref{eqn:GG.Straighten} implies that $t_1^a t_2^b(1-t_1)$ is $\GG_\bfst$-equivalent to $- t_1^{b-1}t_2^{a+1}(1-t_1)$, and hence the first and last terms in $\tilde\sigma$ cancel. One notes also that the second and second-to-last terms similarly cancel, \emph{et cetera}. Hence all the terms cancel in pairs, except if $\tilde{\sigma}$ has an odd number of terms (occurring whenever $b-a$ is odd). In this case, the uncanceled term in the middle is necessarily of the form $t_1^{a+r}t_2^{a+r+1}(1-t_1)$ where $r=(b-a-1)/2$ and \eqref{eqn:GG.Straighten} implies this is $\GG$-equivalent to zero.
\end{proof}

Our notions of \emph{cancelling segments} and the sets $\Seg_m(I)$ are related by the following proposition.

\begin{prop}
\label{prop:Seg.Cancel}
Let $I\in\Bad_{k,m}$ and fix $\epsilon\in\{0,1\}^{k-1}$ with $\epsilon_m =1$. Then
\begin{equation}
\label{eqn:Seg.Cancel}
\sum_{J\in \Seg_m(I)} \T{\lambda}{J}{\epsilon}
\end{equation}
forms a cancelling segment of Type A, B, or C on the variables $t_m$ and $t_{m+1}$. Specifically,
	\begin{enumerate}[label=\emph{(\alph*)},leftmargin=*]
		\item if $k=m+1$ and $\lambda_m = \lambda_{m+1}$, then \eqref{eqn:Seg.Cancel} is of Type A and hence is $\GG$-equivalent to $0$;
		\item conversely if $k>m+1$ or $\lambda_m>\lambda_{m+1}$, when
			\begin{enumerate}[label=\emph{(\roman*)}]
				\item $\epsilon_{m+1}=0$, then \eqref{eqn:Seg.Cancel} is of Type B and hence is $\GG$-equivalent to $\T{\lambda}{\Max_m(I)}{\epsilon}$;
				\item $\epsilon_{m+1}=1$, then \eqref{eqn:Seg.Cancel} is of Type C and hence is $\GG$-equivalent to $\T{\lambda}{\Max_m(I)}{\epsilon'}$ where $\epsilon' = \epsilon - e_{m+1}$. 
			\end{enumerate}
	\end{enumerate}
\end{prop}

\begin{proof}
Recall that $I\in \Bad_{k,m}$ means that $I_m=0$ and either $\lambda_m\leq\lambda_{m+1}+I_{m+1}$ (for $\epsilon_m=1$) or $\lambda_{m}<\lambda_{m+1}+I_{m+1}$ (for $\epsilon_m=0$). Throughout the proof, write $a=\lambda_m$ and $b=\lambda_{m+1}+I_{m+1}$. 

We first consider (a). In this case, $k=m+1$ necessarily implies that $\epsilon_{m+1} = 0$, so we observe that the sum \eqref{eqn:Seg.Cancel} is equal to
\begin{equation}
\label{eqn:Seg.case.a}
f\cdot \left(t_m^{a}t_{m+1}^{b} + t_m^{a+1}t_{m+1}^{b-1} + \cdots  + t_m^{\lambda_{m}+[I_{m+1}-(\lambda_m - \lambda_{m+1} +1)]}t_{m+1}^{\lambda_{m+1} + [\lambda_{m}-\lambda_{m+1} +1]}\right)(1-t_m)
\end{equation}
since the requirement $k=m+1$ for all $J\in\Seg_m(I)$ and the condition $\lambda_m = \lambda_{m+1} = a$ implies that the sequence in $J\in\Seg_m(I)$ with least value of $J_{m+1}$ must have $J_{m+1} = \lambda_m -\lambda_{m+1} +1 = 1$. Moreover, in the expression above we have set
\[
f = \prod_{i\neq m,m+1} t_i^{\lambda_i+I_i} \cdot \prod_{i=1}^{m-1} (1-t_i)^{\epsilon_i}
\]
which is independent of $t_m$ and $t_{m+1}$. Now, arithmetic in the exponents of \eqref{eqn:Seg.case.a} implies that \eqref{eqn:Seg.Cancel} is equal to $f\cdot \tilde\sigma_m(a,b)\cdot(1-t_m)$, a cancelling segment of Type A, as desired.

In case (b), the sum \eqref{eqn:Seg.Cancel} becomes:
\begin{equation}
\label{eqn:Seg.case.b}
f\cdot \left(t_m^{a}t_{m+1}^{b} + t_m^{a+1}t_{m+1}^{b-1} + \cdots + t_m^{\lambda_{m}+[I_{m+1}-(\lambda_m - \lambda_{m+1})]}t_{m+1}^{\lambda_{m+1} + [\lambda_{m}-\lambda_{m+1}]}\right)(1-t_m)(1-t_{m+1})^{\epsilon_{m+1}}
\end{equation}
where we now have
\[
f = \prod_{i\neq m,m+1} t_i^{\lambda_i+I_i} \cdot \prod_{i\neq m,m+1} (1-t_i)^{\epsilon_i},
\]
which is again independent of $t_m$ and $t_{m+1}$. We note that, taking the last term of the parenthetical sum in \eqref{eqn:Seg.case.b} produces
\[
\T{\lambda}{\Max_m(I)}{\epsilon} = f\cdot \left(t_m^{\lambda_{m}+[I_{m+1}-(\lambda_m - \lambda_{m+1})]}t_{m+1}^{\lambda_{m+1} + [\lambda_{m}-\lambda_{m+1}]}\right)(1-t_m)(1-t_{m+1})^{\epsilon_{m+1}}.
\]
Hence, $\T{\lambda}{\Max_m(I)}{\epsilon} = f\cdot \left(t_m^b t_{m+1}^a\right) (1-t_m)(1-t_{m+1})^{\epsilon_{m+1}}$ and so \eqref{eqn:Seg.case.b} is equal to $f\cdot\sigma_m(a,b)\cdot(1-t_m)(1-t_{m+1})^{\epsilon_{m+1}}$, which is either a cancelling segment of Type B (when $\epsilon_{m+1} = 0$) or of Type C (when $\epsilon_{m+1} = 1$). The specified $\GG_\bfst$-equivalences then follow from Proposition \ref{prop:cancel.segments}.
\end{proof}

The following Lemma establishes that the results of performing the cancellations described by Proposition \ref{prop:Seg.Cancel} result in $m$-sorted polynomials. Of course, in Proposition \ref{prop:Seg.Cancel}(a) the resulting $\GG_\bfst$-equivalent polynomial is trivially sorted since it vanishes.

\begin{lem}
\label{lem:sorted.after.cancel}
The polynomials $\T{\lambda}{\Max_m(I)}{\epsilon}$ of case \emph{(b.i)} and $\T{\lambda}{\Max_m(I)}{\epsilon'}$ of case \emph{(b.ii)} of Proposition \ref{prop:Seg.Cancel} are both $m$-sorted.
\end{lem}

\begin{proof}
When $\epsilon_{m+1} = 0$ we have 
\begin{align}
\T{\lambda}{\Max_m(I)}{\epsilon} 
	&= f(t)\cdot t_m^{\lambda_m+[I_{m+1}-(\lambda_m-\lambda_{m+1})]} t_{m+1}^{\lambda_{m+1}+[\lambda_m-\lambda_{m+1}]} \cdot (1-t_m)^{\epsilon_m} 
	\nonumber \\
	&= f(t)\cdot t_m^{I_{m+1}+\lambda_{m+1}}t_{m+1}^{\lambda_m} \cdot (1-t_m)^{\epsilon_m}
	\label{eqn:star}
\end{align}
where ${I_{m+1}+\lambda_{m+1}} > {\lambda_m}$ is guaranteed by the assumption that $I\in\Bad_{k,m}$, i.e.~$I$ is not $m$-sorted and $I_m = 0$. Hence \eqref{eqn:star} is $m$-sorted.

When $\epsilon_{m+1}=1$ we have that the expression \eqref{eqn:star} is now equal to $\T{\lambda}{\Max_m(I)}{\epsilon'}$. Moreover, this time ${I_{m+1}+\lambda_{m+1}} \geq {\lambda_m}$ by the assumption that $I\in\Bad_{k,m}$. Since the exponent on $t_{m+1}$ will be identical for every monomial term in the expansion of \eqref{eqn:star}, this is still enough to guarantee that the polynomial is $m$-sorted.
\end{proof}

We conclude this subsection with two Lemmas that will be used repeatedly in Section \ref{ss:K.Pieri.positivity.proof.algorithm} to guarantee that the cancellation of non-sorted terms in Theorem \ref{thm:GG.finiteness} is organized appropriately. In particular, Lemma \ref{lem:Seg.disjoint} implies that the needed canceling segments are always disjoint; Lemma \ref{lem:Seg.constant.epsilon} guarantees that the value of $\epsilon$ will be constant on our canceling segments so that Proposition \ref{prop:Seg.Cancel} applies.

\begin{lem}
\label{lem:Seg.disjoint}
If $I,I'\in\Bad_{k,m}$ then either $\Seg_m(I)\intersect\Seg_m(I') = \emptyset$ or $I=I'$.
\end{lem}

\begin{proof}
Suppose that $J\in\Seg_m(I)\intersect\Seg_m(I')$. Then $I_i = J_i = I'_i$ for all $i\neq m,m+1$. Moreover, $I_m=I'_m = 0$ and so the condition that $|I|=n=|I'|$ also implies that $I_{m+1} = I'_{m+1}$; i.e.~$I=I'$.

Conversely, if $I\neq I'$, then $I_i\neq I'_i$ for at least one $i\neq m$ (we already must have $I_m = 0 = I'_m$). If $i\neq m+1$ then every element of $\Seg_m(I)$ must have a different $i$-th entry than any element of $\Seg_m(I')$, and hence their intersection is empty. Finally, if $i=m+1$ then suppose $J$ is in the intersection. By the definition of $\Seg_m$ we must have $I_j = J_j = I'_j$ for every $j\neq m,m+1$. But then $I_m = 0 = I'_m$ and $|I| = n = |I'|$ contradicts the assumption that $I_{m+1} \neq I'_{m+1}$.
\end{proof}

\begin{lem}
\label{lem:Seg.constant.epsilon}
Let $I'\in\Bad_{k,m}$. Suppose we are in the case $k>m+1$ or $\lambda_m>\lambda_{m+1}$ and let $I = \Max_m(I')$. 
If $I\in\Bad_{k,m'}$ for $m'<m$ then for every $J\in\Seg_{m'}(I)$ there exists $J'\in\Bad_{k,m}$ such that $J=\Max_m(J')$.
\end{lem}

\begin{proof}
First, we observe the following fact: for any $K'\in\Bad_{k,m}$ we have that $K=\Max_m(K')$ if and only if $K\in\Seg_m(K')$ and $K_{m+1} = \lambda_m-\lambda_{m+1}$.

Applying this observation to $I$, we see that since $I=\Max_m(I')$ we must have $I_m = I'_{m+1}-(\lambda_m-\lambda_{m+1})$ and $I_{m+1} = \lambda_m - \lambda_{m+1}$. Since $I\in\Bad_{k,m'}$ we must also have $I_{m'}=0$. Thus $I$ has the form:
\[
(I_1,\ldots,I_{m'-1},0,I_{m'+1},\ldots,I_{m-1},I'_{m+1}-(\lambda_m-\lambda_{m+1}), \lambda_m-\lambda_{m+1},I_{m+2},\ldots,I_{p+1}).
\]
Now suppose that $J\in\Seg_{m'}(I)$. This means that $J_i = I_i$ for all $i\neq m',m'+1$. In particular, $m>m'$ means that $m+1>m'+1$ and thus $J_{m+1} = \lambda_m - \lambda_{m+1}$.

Now define a sequence $J'$ by: $J'_j = J_j$ for all $j\neq m,m+1$, $J'_m = 0$, and $J'_{m+1} = J_m + J_{m+1}$. Indeed, $|J'|=|J|$ so $J'\in\Comp_{p+1}(n)$ and moreover 
\begin{align*}
\lambda_{m+1} + J'_{m+1} 
	& = \lambda_{m+1}+J_m + J_{m+1} \\
	& = \lambda_{m+1}+J_m + (\lambda_m - \lambda_{m+1}) \\
	& = J_m + \lambda_m \geq J'_m + \lambda_m
\end{align*}
where the last inequality follows because $J'_m = 0$. This proves $\T{\lambda}{J'}{}$ is not $m$-sorted and thus $J'\in\Bad_{k,m}$. Finally, the fact that $J_{m+1} = \lambda_m - \lambda_{m+1}$ guarantees that $J = \Max_m(J')$.
\end{proof}

\subsection{Positivity in the $K$-Pieri rule}
\label{ss:K.Pieri.positivity.proof.algorithm}

For a fixed integer $k\in[p+1]$, we consider the set of polynomials $\curly T_k = \{\T{\lambda}{I}{}:I\in\Comp_{p+1}^k(n)\}$. We will now describe an algorithmic procedure which proves the following theorem.

\begin{thm}
\label{thm:equivalent.to.sorted.poly}
For every $k \in [p+1]$, the sum $\sum_{f\in\curly T_k} f$ is $\GG_\bfst$-equivalent to a sorted polynomial.
\end{thm}

\begin{proof}
We begin with the special case that $\lambda_{k-1} = \lambda_k$. We claim that $\sum_{f\in\curly T_k} f$ is $\GG_\bfst$-equivalent to zero, and hence is trivially sorted. In this case, take $I = (I_1,\ldots,I_k,0,\ldots,0)$ with $I_k\neq 0$; that is, $I\in \Comp_{p+1}^k(n)$. There exists a sequence
\[
I' = (I_1,\ldots,I_{k-2},0,I_{k-1} + I_k,0,\ldots,0)
\]
also in $\Comp_{p+1}^k(n)$. Moreover, since $\lambda_{k-1} = \lambda_k$, it is guaranteed that $I'\in\Bad_{k,k-1}$. Thus Proposition \ref{prop:Seg.Cancel}(a) implies that
\[
\GG_\bfst\left( \sum_{J\in\Seg_{k-1}(I')} \T{\lambda}{J}{}\right) = 0.
\]
Since $I\in\Seg_{k-1}(I')$ and this can be done for any $I \in \Comp_{p+1}^k(n)$, we have proven that every term of $\curly T_k$ appears as part of a cancelling segment of Type A as above. Finally, Lemma \ref{lem:Seg.disjoint} ensures that each term of $\curly T_k$ appears in exactly one such segment, and the claim is proved.

If, on the other hand, $\lambda_{k-1} > \lambda_{k}$, we will use an inductive procedure. By Lemma \ref{lem:m.bound} the largest $m$ for which a term $f\in \curly T_k$ can fail to be $m$-sorted is $m = k-1$. For $I\in \Bad_{k,k-1}$ form the cancelling segment
\[
\sum_{J\in\Seg_{k-1}(I)} \T{\lambda}{J}{}.
\]
By Proposition \ref{prop:Seg.Cancel}(b), this is $\GG_\bfst$-equivalent to $\T{\lambda}{\Max_{k-1}(I)}{}$ and furthermore, by Lemma \ref{lem:sorted.after.cancel}, $\T{\lambda}{\Max_{k-1}(I)}{}$ is $(k-1)$-sorted. Since Lemma \ref{lem:Seg.disjoint} ensures that such cancelling segments are disjoint, after performing these cancellations for all $I\in\Bad_{k,k-1}$ we are left only with polynomials $\T{\lambda}{J}{} \in \curly T_k$ which are $(k-1)$-sorted. We will repeat a similar procedure for each $m<k-1$, in descending order. 

Consider a fixed $m<k-1$ and suppose we have completed our cancelling procedure corresponding to integers exceeding $m$. This implies that any remaining uncanceled $\T{\lambda}{I}{\epsilon}$, with $I\in\Comp_{p+1}^k(n)$, is $m'$-sorted for $m'>m$. If all of these are already $m$-sorted, then go on to consider those which are not $(m-1)$-sorted. Otherwise, Remark \ref{rem:Bad.empty} guarantees $\Bad_{k,m}$ is non-empty, and so we choose $I \in \Bad_{k,m}$. We form the cancelling segment
\begin{equation}
\label{eqn:big.proof.cancel.seg.m}
\sum_{J\in\Seg_{m}(I)} \T{\lambda}{J}{\epsilon}
\end{equation}
where the value of $\epsilon$ depends on needed cancellations from steps corresponding to $m'>m$. Explicitly, if for some $J\in\Seg_{m}(I)$ we had $J = \Max_{m'}(J')$ for some $J'$, then $\epsilon_{m'+1}$ will be $0$. In any event, Lemma \ref{lem:Seg.constant.epsilon} guarantees that the value of $\epsilon$ will be constant on $\Seg_{m}(I)$ and moreover, since we are working by descending values of $m$, we will have $\epsilon_m = 1$. Hence Proposition \ref{prop:Seg.Cancel}(b) implies that \eqref{eqn:big.proof.cancel.seg.m} is $\GG_\bfst$-equivalent to either
\begin{enumerate}[label=(\roman*)]
\item $\T{\lambda}{\Max_{m}(I)}{\epsilon}$ when $\epsilon_{m+1}=0$, or
\item $\T{\lambda}{\Max_{m}(I)}{\epsilon'}$ (with $\epsilon' = \epsilon-e_{m+1}$) when $\epsilon_{m+1}=1$.
\end{enumerate}
Using Lemma \ref{lem:sorted.after.cancel}, we conclude that the result is $m$-sorted, while Lemma \ref{lem:Seg.disjoint} implies the segments above for each $I \in \Bad_{k,m}$ are disjoint. Hence, after performing this cancellation for each $I\in \Bad_{k,m}$,  we are left only with polynomials $\T{\lambda}{J}{\epsilon}$ which are $m'$-sorted for all $m'\geq m$.

By repeating this procedure on decreasing values of $m$, we finally obtain a polynomial which is sorted, as desired.
\end{proof}

\begin{cor}
\label{cor:GG.positivity}
The coefficients $c_{\lambda,n}^\mu$ of Theorem \ref{thm:K.Pieri.rule} are positive.
\end{cor}

\begin{proof}
The algorithm described by the proof of Theorem \ref{thm:equivalent.to.sorted.poly} implies that the product $G_\lambda\cdot G_{(n)}$ is equal to applying $\GG_\bfst$ to a sum of sorted polynomials $\T{\lambda}{I}{\epsilon}$. The lowest degree term of $\T{\lambda}{I}{\epsilon}$ is $t^{\lambda + I}$, whose degree is equal to $|\lambda| + n$. Higher degree terms are obtained by expanding \[\prod_{i=1}^{\ell(I)-1} (1-t_i)^{\epsilon_i}\] from which it follows that the signs of $\T{\lambda}{I}{\epsilon}$ alternate in total degree, as specified in Theorem \ref{thm:K.Pieri.rule}.
\end{proof}

\section{Translation to combinatorics}
\label{s:translate.comb}

\subsection{A new statement of the $K$-Pieri rule}
In this subsection, we give a combinatorial translation for the $K$-Pieri rule which follows from the positivity proof of Section \ref{s:K.Pieri.positivity}.

\begin{thm}
\label{thm:GG.K.Pieri.comb}
For any partition $\lambda$ and positive integer $n$, we have
\begin{equation}
\label{eqn:GG.K.Pieri.comb}
G_\lambda\cdot G_{(n)} = \GG_\bfst \left(
	\sum_\mu t^\mu \prod_{i=1}^p(1-t_i)^{\epsilon_i(\mu)} \right)
\end{equation}
where the sum ranges over partitions $\mu\supset\lambda$ such that $\mu/\lambda$ is a horizontal $n$-strip. The numbers $\epsilon_i(\mu)\in\{0,1\}$ are determined by the following rules:
	\begin{enumerate}[leftmargin=*,label=\emph{(\alph*)}]
	\item $\epsilon_1(\mu) = 1$ if and only if $\ell(\mu-\lambda)>1$;
	\item when $i \geq \ell(\mu-\lambda)$ then $\epsilon_i(\mu) = 0$;
	\item when $1 < i < \ell(\mu-\lambda)$ then $\epsilon_i(\mu) = 0$ if and only if $\lambda_{i-1}=\mu_i$.
	\end{enumerate}
Each of the polynomials $t^\mu \prod_{i=1}^p(1-t_i)^{\epsilon_i(\mu)}$ is obtained through the algorithmic procedure of Section \ref{ss:K.Pieri.positivity.proof.algorithm} and is necessarily sorted.
\end{thm}

\begin{proof}
First observe that $t^\mu \prod_{i=1}^p(1-t_i)^{\epsilon_i(\mu)} = \T{\lambda}{I}{\epsilon(\mu)}$ where $I = \mu-\lambda$. If a polynomial $\T{\lambda}{I}{\epsilon}$ survives our cancellation procedures for some $\epsilon$, then $\mu=\lambda+I$ must necessarily be a partition obtained by adding $n$ boxes to $\lambda$ (since $|I|=n$). We now show that if two of the added boxes appear in the same column, then $\T{\lambda}{I}{\epsilon}$ must be cancelled at some point in our procedure (no matter the value of $\epsilon$).

The condition that two boxes are added to the same column of $\lambda$ in order to form the partition $\mu$ is equivalent to the existence of an $i\in[p+1]$ such that
\begin{equation}
\label{eqn:same.column.cond}
\lambda_{i}+I_{i}>\lambda_{i-1}\quad\text{and}\quad I_{i-1}>0,
\end{equation}
or, as pictures of the relevant rows of the Young diagram we have the following (with the shaded portion indicating that boxes have been added to the same column).
\begin{center}
\begin{tikzpicture}[scale=.4]
\fill[cyan!25] (3,0) rectangle (4,2); %same column overlap

\draw (0,0) rectangle (2,1); %\lambda_i
\draw (0,1) rectangle (3,2); %\lambda_{i-1}
\draw (3,1) rectangle (5,2); %I_{i-1}
\draw (2,0) rectangle (4,1); %I_i

\node (lambdai) at (1,.5){\tiny$\lambda_i$};
\node (lambdai-1) at (1.5,1.5){\tiny$\lambda_{i-1}$};
\node (Ii) at (3,.5){\tiny$I_i$};
\node (Ii-1) at (4,1.5){\tiny$I_{i-1}$};

\end{tikzpicture}
\end{center}
Form the sequence $I' \in \Comp_{p+1}(n)$ as follows:
\[
I' = \left( I_1, \ldots, I_{i-2}, 0, I_{i-1} + I_i, I_{i+1}, \ldots, I_{p+1} \right).
\]
and set $k = \ell(I')$. The conditions \eqref{eqn:same.column.cond} imply that $I'\in\Bad_{k,i-1}$. Moreover, we observe that $\ell(I) = \ell(I')$ and $\lambda_{i-1}-\lambda_i<I_{i-1}$ which implies that $I \in \Seg_{i-1}(I')$. In any case of Proposition \ref{prop:Seg.Cancel}, $I \neq \Max_{k-1}(I')$ and hence the term $\T{\lambda}{I}{\epsilon}$ will be cancelled by our procedure.

Now, we wish to show that if $\mu = \lambda+I$ for a sequence $I$ which does \emph{not} add two boxes in \emph{any} column, then $\T{\lambda}{I}{\epsilon}$ survives our procedure, for some value of $\epsilon$. Set $\ell(I) = k$. In this case, $I$ must have the property that for all $i\in[k]$,
\begin{equation}
\label{eqn:no.column.cond}
\lambda_i + I_i \leq \lambda_{i-1} \iff I_{i}-(\lambda_{i-1}-\lambda_i)\leq 0
\end{equation}
or as a picture:
\begin{center}
\begin{tikzpicture}[scale=.4]

\draw (0,0) rectangle (1.5,1); %\lambda_i
\draw (0,1) rectangle (3,2); %\lambda_{i-1}
\draw (3,1) rectangle (5,2); %I_{i-1}
\draw (1.5,0) rectangle (3,1); %I_i

\node (lambdai) at (.75,.5){\tiny$\lambda_i$};
\node (lambdai-1) at (1.5,1.5){\tiny$\lambda_{i-1}$};
\node (Ii) at (2.25,.5){\tiny$I_i$};
\node (Ii-1) at (4,1.5){\tiny$I_{i-1}$};

\end{tikzpicture}
\end{center}
If $\T{\lambda}{I}{\epsilon}$ is cancelled by our procedure then there must exist $m$ and $\hat{I}$ such that $I\in \Seg_{k,m}(\hat{I})$. Recall that by Lemma \ref{lem:m.bound} we must have that $k\geq m+1$. Moreover, we must have
\[
\hat{I} = \left( I_1,\ldots, I_{m-1}, 0 , I_{m}+I_{m+1}, I_{m+2} ,\ldots, I_k, 0,\ldots, 0\right).
\]
We claim that we cannot be in the case that $k=m+1$ and $\lambda_m=\lambda_{m+1}$, for otherwise $I$ does not satisfy the condition that $\lambda_{m}-\lambda_{m+1} \geq I_{m+1}$ since the lefthand side is zero while the righthand side is positive. Hence, we consider the cancelling segment (of Type B or C), $\sum_{J\in\Seg_m(\hat{I})} \T{\lambda}{J}{\epsilon}$. 
By definition, $\Max_m(\hat{I})$ equals
\begin{equation}
\label{eqn:Max.application}
	(I_1,\ldots,I_{m-1},
			I_m+ \mathop{\underbrace{I_{m+1} - (\lambda_m-\lambda_{m+1})}}_{\leq 0~\text{by \eqref{eqn:no.column.cond}}},
			\lambda_m-\lambda_{m+1},
			I_{m+2},\ldots,I_k,0,\ldots,0
			).
\end{equation}
If the underbraced inequality is strict then, because $I_m$ exceeds the $m$-th component of $\Max_m(\hat{I})$, we conclude $I$ is not in $\Seg_m(\hat{I})$ and therefore $\T{\lambda}{I}{\epsilon}$ is not cancelled. If, on the other hand, the underbraced inequality is zero, then $I = \Max_m(\hat{I})$ and hence $\T{\lambda}{I}{\epsilon}$ is also not cancelled, but either 
\[ 
\T{\lambda}{I}{\epsilon} = \T{\lambda}{\Max_m(\hat{I})}{\epsilon}
\quad\text{ or }\quad 
\T{\lambda}{I}{\epsilon} \rightsquigarrow \T{\lambda}{\Max_m(\hat{I})}{\epsilon'}
\]
depending on the value of $\epsilon_{m+1} \in \{0,1\}$, c.f.~Proposition \ref{prop:Seg.Cancel}. This establishes that $G_\lambda\cdot G_{(n)}$ can be written as 
\[
\GG_\bfst\left(\sum_\mu t^\mu\prod_{i=1}^p(1-t_i)^{\epsilon_i(\mu)}\right)
\] 
with the sum over the desired $\mu$. It remains only to establish the conditions (a), (b), and (c) on the numbers $\epsilon_i(\mu)$.

We observe that in our procedure from Section \ref{ss:K.Pieri.positivity.proof.algorithm} the value of $\epsilon_i$ becomes zero only when cancelling terms which are not $(i-1)$-sorted. Hence for (a), if $\ell(I)>1$, then $\epsilon_1$ is never changed to zero. If $\ell(I)=1$, which happens only for the sequence $(n,0,\ldots,0)$, then $\T{\lambda}{(n,0,\ldots,0)}{} = t^{\lambda+I}$. Condition (b) follows directly from the definition of $\T{\lambda}{I}{\epsilon}$.

To prove (c), notice that the value of $\epsilon_i(\mu)$ changes $1 \mapsto 0$ if and only if $I = \mu-\lambda = \Max_{i-1}(\hat{I})$ for some $\hat{I}$.  We can see, e.g.~from Equation \eqref{eqn:Max.application}, that this occurs exactly when $\lambda_{i-1}-\lambda_{i} = I_i = \mu_i-\lambda_i$; i.e.~when $\lambda_{i-1} = \mu_i$ as desired.
\end{proof}

\begin{defn}
\label{defn:A.lambda.mu}
Given a partition $\mu\supset \lambda$, such that $\mu/\lambda$ is a horizontal $n$-strip, define the set 
\[
A_{\lambda,\mu} = \left\{
	i\in[p] : \epsilon_i(\mu) = 1
	\right\}.
\]
\end{defn}

With this notation, we can restate Theorem \ref{thm:GG.K.Pieri.comb} in the more aesthetically pleasing form below.

\begin{cor}
\label{cor:GG.K.Pieri.comb.Asets}
For any partition $\lambda$ and positive integer $n$, we have
\begin{equation}
\label{eqn:GG.K.Pieri.comb.Asets}
G_\lambda\cdot G_{(n)} = \GG_\bfst \left(
	\sum_\mu t^\mu \prod_{i\in A_{\lambda,\mu}} (1-t_i) \right)
\end{equation}
where the sum ranges over partitions $\mu$ such that $\mu/\lambda$ is a horizontal $n$-strip.
\end{cor}

Furthermore, we can give a combinatorial description of the sets $A_{\lambda,\mu}$. It is known that $\mu/\lambda$ is a horizontal strip if and only if $\mu$ and $\lambda$ satisfy the following ``interlacing'' property, see e.g.\ \cite[Section~I.1]{im1995},
\begin{equation}
\label{eqn:interlaced}
\mu_1\geq\lambda_1\geq\mu_2\geq\lambda_2 \geq \cdots
\end{equation}
Thus every horizontal strip $\mu/\lambda$ corresponds to a sequence of strict inequalities and equalities (with only finitely many inequalities). We will call this sequence the \emph{code} of the horizontal strip, and we will denote it by $\code(\mu/\lambda)$. From this, we define the \emph{odd code} (respectively \emph{even}) \emph{code} which is the subsequence comprised of the terms of $\code(\mu/\lambda)$ indexed by odd (resp.\ even) numbers; by convention, the first term of a sequence is indexed by $1$. These will be respectively denoted by $\ocode(\mu/\lambda)$ and $\ecode(\mu/\lambda)$.

\begin{example}
\label{ex:code}
For example, with $\mu=(3,3,2,1,1)$ and $\lambda=(3,2,2,1)$ the skew diagram $\mu/\lambda$ is the horizontal $2$-strip (shaded blue)
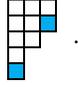
\[
\ytableausetup{centertableaux,smalltableaux}
{\tiny
\ydiagram [*(cyan)]{3+0,2+1,2+0,1+0,0+1} * [*(white)]{3,3,2,1,1}
}.
\]
The interlacing property of Equation \eqref{eqn:interlaced} takes the form
\[
3 = 3 = 3 > 2 = 2 = 2 > 1 = 1 = 1 > 0 = \cdots
\]
and the corresponding code is
\[
\code\left( \mu/\lambda \right) = \{ = , = , > , = , = , >, =,=,>,=,\ldots\}. 
\]
Hence,
\begin{gather*}
\ocode\left( \mu/\lambda \right) = \{ =  , >  , = , = , > , = ,\ldots\} \\
\ecode\left( \mu/\lambda \right)  =\{ = , = , >, = ,\ldots\}.
\end{gather*}
\end{example}

This leads to the following combinatorial description of the sets $A_{\lambda,\mu}$.

\begin{prop}
\label{prop:A.lambda.mu.code}
If the last non-equality appearing in $\ocode(\mu/\lambda)$ occurs in its $k$-th term, then $\ell(\mu-\lambda) = k$. Furthermore
	\begin{enumerate}[leftmargin=*,label=\emph{(\alph*)}]
	\item $1\in A_{\lambda,\mu}$ if and only if $k>1$;
	\item $i\notin A_{\lambda,\mu}$ for all $i\geq k$;
	\item for $1<i<k$, $i\in A_{\lambda,\mu}$ if and only if the $(i-1)^{\text{st}}$ entry of $\ecode(\mu/\lambda)$ is a strict inequality.
	\end{enumerate}
\end{prop}

\begin{proof}
The $i^{\text{th}}$ entry of the odd code compares $\mu_i$ to $\lambda_i$. For large enough $i$, both of these numbers are zero and are therefore equal. Hence, the last occurrence of a ``$>$'' sign in the odd code must be the length of the sequence $\mu-\lambda$. Properties (a), (b), and (c) are the straightforward translations of items (a), (b), and (c) from Theorem \ref{thm:GG.K.Pieri.comb}.
\end{proof}	

\subsection{Cohomological Pieri rule}
\label{ss:coho.pieri.rule.pf}
We can apply our iterated residue version of the $K$-Pieri rule to prove the cohomological Pieri rule. 

\begin{cor}
\label{cor:GG.implies.H.Pieri}
The result of Theorem \ref{thm:GG.K.Pieri.comb} implies Theorem \ref{thm:H.Pieri.rule}, a.k.a.\ the cohomological Pieri rule.
\end{cor}

\begin{proof}
For a partition $\nu$, the lowest degree homogeneous part of $G_\nu$ is the Schur function $s_\nu$, so applying the observations of Remark \ref{rem:low.deg.of.GG} to Equation \eqref{eqn:GG.K.Pieri.comb.Asets} gives
\[
s_\lambda\cdot s_{(n)} = \SSS_\bfst\left(\sum_\mu t^\mu\right) = \sum_\mu s_\mu
\] 
with the sum ranging over partitions $\mu$ such that $\mu/\lambda$ is a horizontal $n$-strip.
\end{proof}

\subsection{An example computation}
We conclude this section with an example. Let $\lambda = (3,2,2,1)$ and $n=2$. The results of applying our method to the multiplication $G_\lambda\,G_{(n)}$ are summarized in Table \ref{table:example.GG.comb}. Furthermore, we illustrate the details of the algorithm in the proof of Theorem \ref{thm:equivalent.to.sorted.poly} below.
\begin{table}
	\begin{tabular}{c|c|c|c}
$\mu$ & $\ell(\mu-\lambda)$ & $\epsilon(\mu)$ & $A_{\lambda,\mu}$

 \\ \hline &&& \\

{\tiny
\ytableausetup{centertableaux}
\ydiagram [*(cyan)]{3+2,2+0,2+0,1+0} * [*(white)]{5,2,2,1}
} $= (5,2,2,1)$
& $1$
& $\epsilon_i = 0 \quad \forall i\geq 1$
& $\emptyset$

\\ &&& \\ \hline &&& \\

{\tiny
\ytableausetup{centertableaux}
\ydiagram [*(cyan)]{3+1,2+1,2+0,1+0} * [*(white)]{4,3,2,1}
} $= (4,3,2,1)$
& $2$
& {$\begin{array}{c} \epsilon_1 = 1 \\ 
	\epsilon_i = 0 \quad \forall i\geq2 \end{array}$}
& \{1\}

\\ &&& \\ \hline &&& \\

{\tiny
\ytableausetup{centertableaux}
\ydiagram [*(cyan)]{3+1,2+0,2+0,1+1} * [*(white)]{4,2,2,2}
} $= (4,2,2,2)$
& $4$
& {$\begin{array}{c} \epsilon_1 = 1 \\ 
	\mu_2\neq\lambda_1 \implies \epsilon_2 =1 \\
	\mu_3 = \lambda_2 \implies \epsilon_3 = 0 \\
	\epsilon_i = 0 \quad \forall i\geq4 \end{array}$}
& $\{1,2\}$

\\ &&& \\ \hline &&& \\

{\tiny
\ytableausetup{centertableaux}
\ydiagram [*(cyan)]{3+0,2+1,2+0,1+1} * [*(white)]{3,3,2,2}
} $= (3,3,2,2)$
& $4$
& {$\begin{array}{c} \epsilon_1 = 1 \\ 
	\mu_2 = \lambda_1 \implies \epsilon_2 = 0 \\
	\mu_3 = \lambda_2 \implies \epsilon_3 = 0 \\
	\epsilon_i = 0 \quad \forall i\geq4 \end{array}$}
& $\{1\}$

\\ &&& \\ \hline &&& \\

{\tiny
\ytableausetup{centertableaux}
\ydiagram [*(cyan)]{3+0,2+0,2+0,1+1,0+1} * [*(white)]{3,2,2,2,1}
} $= (3,2,2,2,1)$
& $5$
& {$\begin{array}{c} \epsilon_1 = 1 \\ 
	\mu_2 \neq \lambda_1 \implies \epsilon_2 = 1 \\
	\mu_3 = \lambda_2 \implies \epsilon_3 = 0 \\
	\mu_4 = \lambda_3 \implies \epsilon_4 = 0 \\
	\epsilon_i = 0 \quad \forall i\geq5  \end{array}$}
& $\{1,2\}$

\\ &&& \\ \hline &&& \\

{\tiny
\ytableausetup{centertableaux}
\ydiagram [*(cyan)]{3+0,2+1,2+0,1+0,0+1} * [*(white)]{3,3,2,1,1}
} $= (3,3,2,1,1)$
& $5$
& {$\begin{array}{c} \epsilon_1 = 1 \\ 
	\mu_2 = \lambda_1 \implies \epsilon_2 = 0 \\
	\mu_3 = \lambda_2 \implies \epsilon_3 = 0 \\
	\mu_4 \neq \lambda_3 \implies \epsilon_4 = 1 \\
	\epsilon_i = 0 \quad \forall i\geq5  \end{array}$}
& $\{1,4\}$

\\ &&& \\ \hline &&& \\

{\tiny
\ytableausetup{centertableaux}
\ydiagram [*(cyan)]{3+1,2+0,2+0,1+0,0+1} * [*(white)]{4,2,2,1,1}
} $= (4,2,2,1,1)$
& $5$
& {$\begin{array}{c} \epsilon_1 = 1 \\ 
	\mu_2 \neq \lambda_1 \implies \epsilon_2 = 1 \\
	\mu_3 = \lambda_2 \implies \epsilon_3 = 0 \\
	\mu_4 \neq \lambda_3 \implies \epsilon_4 = 1 \\
	\epsilon_i = 0 \quad \forall i\geq5  \end{array}$}
& $\{1,2,4\}$
	
\\ &&& \\ \hline  
	\end{tabular}
\caption{The lefthand column depicts all of the ways to form a partition $\mu$ by adding $2$ boxes to $(3,2,2,1)$, with no two added boxes in the same column (added boxes are shaded blue). Observe that the length of the sequence $I = \mu-\lambda$ is the row number of the southernmost added box. The third column shows the computations for conditions (a), (b), and (c) in Theorem \ref{thm:GG.K.Pieri.comb}.}
\label{table:example.GG.comb}
\end{table}

Since $p=4$ we must consider the set of integer sequences
\begin{multline*}
\Comp_5(2) = \{(2,0,0,0,0)\} \disjointunion \{(0,2,0,0,0),(1,1,0,0,0)\} \disjointunion \{(0,0,2,0,0),(0,1,1,0,0),(1,0,1,0,0)\} \\	
	\disjointunion \{(0,0,0,2,0),(0,0,1,1,0),(0,1,0,1,0),(1,0,0,1,0)\} \\
	 \disjointunion \{(0,0,0,0,2),(0,0,0,1,1),(0,0,1,0,1),(0,1,0,0,1),(1,0,0,0,1)\}
\end{multline*}
where we have already partitioned $\Comp_5(2)$ as the disjoint union of $\Comp_5^k(2)$ for $k\in[5]$. When $k = 1$ we have the associated polynomial
\[
\T{\lambda}{(2,0,0,0,0)}{} = t^{(5,2,2,1)}
\]
which is sorted, so we leave it alone. When $k=2$ we have the associated sum of $\T{\lambda}{I}{}$:
\begin{center}
\begin{tikzpicture}
\node[draw=cyan,ultra thick, rectangle] {$\cancel{t^{(3,4,2,1)}(1-t_1)} + {\color{cyan}t^{(4,3,2,1)}(1-t_1)}$};
\end{tikzpicture}
\end{center}
where we used that the above forms a cancelling segment of Type B on $t_1$ and $t_2$. Throughout the rest of the example, we will adopt the following conventions to make our computation more clear. We use blue (as above) to indicate cancellations performed on segments involving $t_1$ and $t_2$, red for segments involving $t_2$ and $t_3$, green for segments involving $t_3$ and $t_4$, and purple for segments involving $t_4$ and $t_5$. The rightmost term in a cancelling segment box corresponding to a sequence of the type $\Max_\bullet(\bullet)$ survives a cancellation of Type B or Type C, and we will color it with the corresponding color. Furthermore, if the segment is of Type C, we will show the corresponding value of $\epsilon_i$ which has become $0$ by crossing out $(1-t_i)$. In the case $k=3$ we obtain
\begin{center}
\begin{tikzpicture}
\node[draw=red,ultra thick,rectangle] at (-4.1,0)
{$\cancel{t^{(3,2,4,1)}(1-t_1)(1-t_2)} + \cancel{t^{(3,3,3,1)}(1-t_1)(1-t_2)}$};
\node at (0,0) {$+$};
\node[draw=red,ultra thick,rectangle] at (2.1,0)
{$\cancel{t^{(4,2,3,1)}(1-t_1)(1-t_2)}.$};
\node at (-4.1,-.75) {\color{red}\bfseries Type A};
\node at (2.1,-.75) {\color{red}\bfseries Type A};
\end{tikzpicture}
\end{center}
where every term $\T{\lambda}{I}{}$ with $I\in\Comp_3(2)$ has appeared in a cancelling segment of Type A, as required by Proposition \ref{prop:Seg.Cancel}(a). For $k=4$, we begin with cancellations for $t_3$ and $t_4$
\begin{center}
\begin{tikzpicture}
\node at (0,1.75){\color{ForestGreen}\bfseries Type B};
\node[draw=ForestGreen,ultra thick,rectangle] at (0,1)
{$\cancel{t^{(3,2,2,3)}(1-t_1)(1-t_2)(1-t_3)} + {\color{ForestGreen}t^{(3,2,3,2)}(1-t_1)(1-t_2)(1-t_3)}$};
\node at (0,0)
{$ + t^{(3,3,2,2)}(1-t_1)(1-t_2)(1-t_3) + t^{(4,2,2,2)}(1-t_1)(1-t_2)(1-t_3).$};
\end{tikzpicture}
\end{center}
where we have not manipulated the last two terms because they are both $3$-sorted and do not correspond to sequences $J\in\Seg_3(I)$ for any $I\in\Bad_{4,3}$.
Now, on the remaining terms we do cancellations for $t_2$ and $t_3$:
\begin{center}
\begin{tikzpicture}
\node[draw=red,ultra thick,rectangle] at (-5.3,0)
{$\cancel{\color{ForestGreen}t^{(3,2,3,2)}(1-t_1)(1-t_2)(1-t_3)} + {\color{red}t^{(3,3,2,2)}(1-t_1)(1-t_2)\cancel{(1-t_3)}.}$};
\node at (0,0) {$+$};
\node[draw=red,ultra thick,rectangle] at (2.7,0)
{${\color{red}t^{(4,2,2,2)}(1-t_1)(1-t_2)\cancel{(1-t_3)}}$};
\node at (-5.3,-.75) {\color{red}\bfseries Type C};
\node at (2.7,-.75) {\color{red}\bfseries Type C};
\end{tikzpicture}
\end{center}
Finally, the middle term above is not $1$-sorted, and so we are left with
\begin{center}
\begin{tikzpicture}
\node[draw=cyan,ultra thick,rectangle] at (-2,0)
{${\color{cyan}t^{(3,3,2,2)}(1-t_1)\cancel{(1-t_2)}}$};
\node at (2,0)
{$+\, t^{(4,2,2,2)}(1-t_1)(1-t_2)$.};

\node at (-2,-.75) {\color{cyan}\bfseries Type C};
\end{tikzpicture}
\end{center}
For $k=5$ we obtain
\begin{center}
\begin{tikzpicture}
\node at (0,1.75) {\color{Purple}\bfseries Type B};
\node[draw=Purple,ultra thick,rectangle] at (0,1)
	{$\cancel{t^{(3,2,2,1,2)}(1-t_1)(1-t_2)(1-t_3)(1-t_4)} + {\color{Purple}t^{(3,2,2,2,1)}(1-t_1)(1-t_2)(1-t_3)(1-t_4)}$};
\node at (0,0)
	{$ +\, t^{(3,2,3,1,1)}(1-t_1)(1-t_2)(1-t_3)(1-t_4) + t^{(3,3,2,1,1)}(1-t_1)(1-t_2)(1-t_3)(1-t_4)$};
\node at (0,-1)
	{$+\, t^{(4,2,2,1,1)}(1-t_1)(1-t_2)(1-t_3)(1-t_4).$};
\end{tikzpicture}
\end{center}
Then on the remainder we have
\begin{center}
\begin{tikzpicture}
\node at (0,1.75) {\color{ForestGreen}\bfseries Type C};
\node[draw=ForestGreen,ultra thick,rectangle] at (0,1)
	{${\color{ForestGreen}t^{(3,2,2,2,1)}(1-t_1)(1-t_2)(1-t_3)\cancel{(1-t_4)}}$};
\node at (0,0)
	{$ +\, t^{(3,2,3,1,1)}(1-t_1)(1-t_2)(1-t_3)(1-t_4)  + t^{(3,3,2,1,1)}(1-t_1)(1-t_2)(1-t_3)(1-t_4)$};
\node at (0,-1)
	{$ +\, t^{(4,2,2,1,1)}(1-t_1)(1-t_2)(1-t_3)(1-t_4),$};
\end{tikzpicture}
\end{center}
and on the above remaining terms we obtain
\begin{center}
\begin{tikzpicture}
\node at (0,1.75) {\color{red}\bfseries Type C};
\node[draw=red,ultra thick,rectangle] at (0,1)
	{${\color{red}t^{(3,2,2,2,1)}(1-t_1)(1-t_2)\cancel{(1-t_3)}}$};
\node[draw=red,ultra thick,rectangle] at (0,0)
	{$+\, \cancel{t^{(3,2,3,1,1)}(1-t_1)(1-t_2)(1-t_3)(1-t_4)}  +  {\color{red}t^{(3,3,2,1,1)}(1-t_1)(1-t_2)\cancel{(1-t_3)}(1-t_4)}$};
\node[draw=red,ultra thick,rectangle] at (0,-1)
	{$ +\, {\color{red} t^{(4,2,2,1,1)}(1-t_1)(1-t_2)\cancel{(1-t_3)}(1-t_4).}$};
\end{tikzpicture}
\end{center}
The remaining middle term is not $1$-sorted, so we perform a final Type C cancellation,
\begin{center}

\begin{tikzpicture}
\node at (0,1) {$t^{(3,2,2,2,1)}(1-t_1)(1-t_2)$};
\node[draw=cyan,ultra thick,rectangle] at (0,0)
	{$+\,{\color{cyan}t^{(3,3,2,1,1)}(1-t_1)\cancel{(1-t_2)}(1-t_4)}$};
\node at (0,-1) {$ +\, t^{(4,2,2,1,1)}(1-t_1)(1-t_2)(1-t_4)$.};
\end{tikzpicture}
\end{center}
In the end, we are left with the following sum of sorted polynomials
\begin{equation}
\label{eqn:ex.final.polynomial}
\begin{split}
t^{(5,2,2,1)} &+ t^{(4,3,2,1)}(1-t_1) + t^{(3,3,2,2)}(1-t_1) \\
	&+ t^{(4,2,2,2)}(1-t_1)(1-t_2) + t^{(3,2,2,2,1)}(1-t_1)(1-t_2) \\ 
	&+ t^{(3,3,2,1,1)}(1-t_1)(1-t_4) + t^{(4,2,2,1,1)}(1-t_1)(1-t_2)(1-t_4)
\end{split}
\end{equation}
which, after applying the $\GG_\bfst$ operator, becomes
\begin{multline}
\label{eqn:ex.final.G.polynomial}
G_{(5,2,2,1)}+G_{(4,3,2,1)} + G_{(3,3,2,2)}+G_{(4,2,2,2)}+G_{(3,2,2,2,1)}+G_{(3,3,2,1,1)}+G_{(4,2,2,1,1)}\\
-G_{(5,3,2,1)}-2 G_{(4,3,2,2)}-G_{(5,2,2,2)}-2 G_{(4,2,2,2,1)} -2 G_{(3,3,2,2,1)}-2 G_{(4,3,2,1,1)}-G_{(5,2,2,1,1)}\\
+G_{(5,3,2,2)}+3 G_{(4,3,2,2,1)}+G_{(5,3,2,1,1)}+G_{(5,2,2,2,1)}-G_{(5,3,2,2,1)}.
\end{multline}

\section{Expansions of single Grothendieck polynomials in the Schur basis}
\label{s:G.poly.S.expand}

Throughout the remainder of the paper, we restrict attention to the \emph{single} stable Grothendieck polynomial, i.e.~$\beta_j=1$ for all $j$ in Equation \eqref{eqn:defn.G}. Moreover, we specialize to the variables $\bfs{x}$ with $x_i = 1-\alpha_i^{-1}$ and fix a positive integer $k$ so that $\bfs{x} = (x_1,\ldots,x_k)$. Then, the polynomial $G_\lambda(\bfs{x})$ is a symmetric function in the $\bfs{x}$ variables and, as such, admits an expansion in the Schur basis. This expansion depends on the number $k$, although in the limit $k\to\infty$ the coefficient of $s_\mu$ (for fixed $\mu$) is stable. 

\subsection{Relating the $\GG_\bfst$ and $\SSS_\bfst$ operators}
\label{ss:GG.relate.SSS}

Following \cite[Section~3]{ab2002.qv}, for each integer $i\geq 0$ define symmetric functions $h^{(i)}_d(\bfs{x})$ via the generating function
\begin{equation}
\label{eqn:higen}
\sum_{d\geq 0} h^{(i)}_d(\bfs{x}) \, u^d = 
\frac{(1-u)^i}{\prod_{j=1}^k (1-x_j u)}.
\end{equation}
The polynomials $h^{(i)}_d$ are non-homogeneous generalizations of the complete homogeneous symmetric functions $h_d$, where in particular, $h^{(0)}_d = h_d$. For any finite integer sequence $I = (I_1,\ldots,I_p)$ with $k\geq p$, define the determinant
\begin{equation}\label{eqn:gIdefn}
g_I(\bfs{x}) = (-1)^{k(k-1)/2}\det\left( h^{(i-1)}_{I_i+j-1}(\bfs{x})\right)_{1\leq i,j\leq k}
\end{equation}
where we agree that $I_r = 0$ for $r>p$. Notice that the size of the determinant depends on the number variables in the alphabet $\bfs{x}$. We have the following relationship to Grothendieck polynomials.

\begin{prop}[\cite{ab2002.qv} Theorem 3.1, \cite{cl2000} Theorem 2.4]
\label{prop:Lenart}
For any partition $\lambda$ with $\ell(\lambda)\leq k$,
\begin{equation}
\label{eqn:Gpoly.det} 
\pushQED{\qed}
G_\lambda(\bfs{x}) =  g_\lambda(\bfs{x}). %\nonumber
\qedhere\popQED
\end{equation}
\end{prop}

\begin{remark}
\label{rem:GI.det.vs.GI.IR}
In \cite[Section~3]{ab2002.qv}, Equation \eqref{eqn:Gpoly.det} is used to define $G_I$ for general integer sequences by replacing $\lambda$ with $I$. As referenced in Section \ref{ss:G.Straighten}, the complexity of this formula depends on $k$, the number of $\bfs{x}$ variables. The complexity of our iterated residue formula \eqref{eqn:defn.G} depends on $p$, the length of the sequence $I$.
\end{remark}

In what follows it is convenient to define another iterated residue operation which encodes data regarding the complete homogeneous symmetric functions $h_d(\bfs{x})$. In particular, for any integer sequence $I=(I_1,\ldots,I_p)$ we define
\begin{equation} \label{eqn:HHdefn}
\HH_{\bfs{t}}\left( {t}^{I} \right)({\bfs{x}})
	= h_I(\bfs{x}) := \prod_{i=1}^p  h_{I_{i}}(\bfs{x}) \nonumber
\end{equation}
and extend the operation to linearly to $\Z[\bfst^{\pm1}]$ just as for the $\SSS_\bfst$ and $\GG_\bfst$ operations. We may omit the reference to the variables $\bfs{x}$ when they are not explicitly needed.

\begin{lem}\label{lem:JTS}
For any Laurent polynomial $f(t_1,\ldots,t_p)$, we have 
\[
\pushQED{\qed}
\SSS_{\bfs{t}}(f) = \HH_{\bfs{t}}\left(f \cdot \prod_{1\leq i<j\leq p} \left(1-\frac{t_i}{t_j}\right)\right). 
\qedhere\popQED 
\]
\end{lem}

\begin{proof}
The result is equivalent to \cite[Lemma~2.5]{rr2013} and \cite[I.(3.4$^{\prime\prime}$)]{im1995}.
\end{proof}

\begin{thm}
\label{thm:GG.to.SSS}
If $\lambda$ is a partition then
\begin{equation}
\label{eqn:GGtoSS}
\GG_{\bfs{t}}({t}^\lambda)(x_1,\ldots,x_k) = \SSS_{\bfs{t}}\left( {t}^\lambda \prod_{i=1}^k (1-t_i)^{i-1}\right)(x_1,\ldots,x_k). 
\end{equation}
\end{thm}

\begin{proof}
Set $p=\ell(\lambda)$. If $k<p$ it is known that both sides are zero, say by using the tableaux descriptions of $G_\lambda(\bfs{x})$ and $s_\lambda(\bfs{x})$ as in \cite{ab2002.klr,im1995}. In the case $k\geq p$,
an application of Lemma \ref{lem:JTS} means the result of the theorem is equivalent to
\begin{equation}
\label{eqn:GGtoHHpartition}
\GG_{\bfs{t}}\left( t^\lambda \right) = \HH_{\bfs{t}}\left( t^\lambda \prod_{1\leq i \leq k}(1-t_i)^{i-1} \prod_{1\leq i<j \leq k}\left( 1 - \frac{t_i}{t_j} \right)\right).
\end{equation}
We will prove the equality of \eqref{eqn:GGtoHHpartition}. Begin by observing that
\[
h^{(i)}_d(\bfs{x}) = \sum_{j=0}^i (-1)^j \binom{i}{j}h_{d-j}(\bfs{x}),
\]
which implies for any positive integer $r$ that
\[
h^{(i)}_d = \HH_{\bfs{t}}\left( t_r^d(1-1/t_r)^i \right) = \HH_{\bfs{t}}\left( t_r^{d-i}(t_r-1)^i\right).
\]
From this we can write the determinant $g_\lambda(\bfs{x})$ of Equation \eqref{eqn:gIdefn} as the iterated residue expression
\begin{equation}
\label{eqn:gIasIR}
g_\lambda(\bfs{x}) = \HH_{\bfs{t}} \left( \det\left( t_i^{\lambda_{i}+j-i}(t_i-1)^{i-1}\right)_{1\leq i,j \leq k}\right)({\bfs{x}})
\end{equation}
where we recall the convention that $\lambda_i = 0$ for $i>p$. Proposition \ref{prop:Lenart} implies that $G_\lambda(\bfs{x}) = g_\lambda(\bfs{x})$, so \eqref{eqn:gIasIR} allows us to compute the lefthand side of \eqref{eqn:GGtoHHpartition} by applying $\HH_{\bfs{t}}$ to the determinant
\[ 
	\begin{vmatrix} 
		t_1^{\lambda_1}						&t_1^{\lambda_1+1}					&\cdots	&t_1^{\lambda_1+k-1} 			\\
		t_2^{\lambda_2-1}(t_2-1)			&t_2^{\lambda_2  }(t_2-1)			&\cdots	&t_2^{\lambda_2+k-2}(t_2-1)	\\
		\vdots								&\vdots								&\ddots	&\vdots							\\
		t_k^{\lambda_k-k+1}(t_k-1)^{k-1}	&t_k^{\lambda_k-k+2}(t_k-1)^{k-1}	&\cdots	&t_k^{\lambda_k}(t_k-1)^{k-1}
	\end{vmatrix}
\]
which can be rewritten as
\[
t^\lambda\prod_{i=1}^{k}t_i^{-i+1}(t_i-1)^{i-1}
	\begin{vmatrix} 
		1		&t_1	&\cdots	&t_1^{k-1} 	\\
		1		&t_2	&\cdots	&t_2^{k-1}	\\
		\vdots	&\vdots	&\ddots	&\vdots		\\
		1		&t_k	&\cdots	&t_k^{k-1}	\\
	\end{vmatrix}.
\]
The \emph{Vandermonde determinant} above is equal to $\prod_{1\leq i<j \leq k}(t_j - t_i)$. Hence, after switching the sign in each factor of $(t_i-1)$ above, one obtains the expression 
\[
	(-1)^{k(k-1)/2}\,
	t^\lambda 	\prod_{i=1}^k (1-t_i)^{i-1}
				\prod_{i=1}^k t_i^{-i+1}
				\prod_{1\leq i<j\leq k}(t_j - t_i)
\]
and this is further equal to
\[
	(-1)^{k(k-1)/2}\,
	t^\lambda \prod_{i=1}^k (1-t_i)^{i-1} \prod_{1\leq i<j \leq k}\left(1 - \frac{t_i}{t_j}\right)
\]
from which the desired equality follows.
\end{proof}

\begin{remark}
The result of Theorem \ref{thm:GG.to.SSS} still holds even when $\lambda$ is not a partition with an additional hypothesis regarding the size $k$ of the alphabet $\bfs{x}$. For more details see the thesis of the first author \cite[Section~3]{ja2014.th}. However, we will not need the more general result in the sequel.
\end{remark}

\begin{remark}\label{rem:whyIR}
Notice that for a Laurent polynomial $f(\bfs{t})$, the result of the operation $\HH_{\bfs t}(f)(\bfs{x})$ is equivalent to taking the constant term (in the $\bfs{t}$ variables) of the formal series
\[
f(\bfs{t}) \prod_{t\in\bfs{t}}\sum_{d\geq 0}\frac{h_d(\bfs x)}{t^d} = \frac{f(\bfs t)}{ \prod_{t\in \bfs t} \prod_{x\in \bfs x}\left( 1-\frac{x}{t}\right)}.
\]
Furthermore, taking the constant term of this series is equivalent to taking the successive residues at $t_1=\infty$, $t_2=\infty$, \emph{et cetera}, of a series with shifted exponents. This observation is another reason for the name \emph{iterated residue operation}. From one point of view, the results of Lemma \ref{lem:JTS} and Theorem \ref{thm:GG.to.SSS} justify this terminology for the $\SSS_\bfst$ and $\GG_\bfst$ operations, respectively. On the other hand, the $\GG_\bfst$ operation already shares an important connection with the residues at both infinity \emph{and} zero via Equation \eqref{eqn:defn.G}. 

The philosophy is that residues at both zero and infinity are required for computations in $K$-theory, but only residues at infinity are required for cohomological computations (c.f.~\cite{as1998,as2003.ver,gbas2012,rr2013,mz2014,ja2014.ir,rras2017}). Because the present context relies heavily on polynomials $s_\lambda$ and $h_d$, we should expect the computations to reflect information in cohomology. In particular, the functions $s_\lambda$ and $h_d$ (evaluated on Chern roots of tautological quotient bundles) are respectively representatives of Schubert and Chern classes in the cohomology of the Grassmannian.
\end{remark}

\begin{remark}
The reader may also notice the resemblance between Lemma \ref{lem:JTS} and the ``raising operator'' formula, c.f.~\cite[I.(3.4$^{\prime\prime}$)]{im1995}. However, the connection between ``exponent to subscript'' operations and contour integration via \eqref{eqn:GGtoSS} is new.
\end{remark}

\subsection{Positivity for Schur expansions of single Grothendieck polynomials}
\label{ss:pos.GG.to.SS}

We now restate another theorem of Lenart and devote this section to proving it with iterated residue techniques.

\begin{thm}[Lenart \cite{cl2000}, Theorem~2.8]
\label{thm:GG.to.SS.alt}
For any partition $\lambda$ and variables $\bfs{x} = (x_1,\ldots,x_k)$ with $k\geq \ell(\lambda)$, one has
\begin{equation}\label{eqn:Schur.Alt}
G_\lambda(\bfs{x}) 
	= \sum_{\lambda\subseteq\mu\subseteq\hat\lambda} 
		(-1)^{|\mu|-|\lambda|} a_{\lambda\mu} s_\mu(\bfs{x})
\end{equation}
where the sum is taken over partitions $\mu$ and each coefficient $a_{\lambda\mu}$ is non-negative. The partition $\hat{\lambda}$ denotes the unique maximal partition of length $k$ obtained by adding at most $j-1$ boxes to the $j$-th row of the Young diagram of $\lambda$.
\end{thm}

We remark that if $k<\ell(\lambda)$ then both sides above are zero. The following lemma plays a role analogous to Theorem \ref{thm:GG.Straighten} but for the $\SSS_\bfst$ operation.

\begin{lem}
\label{lem:SSS.JT}
If $f(\bfs{t})$ is a polynomial symmetric in $t_i$ and $t_{i+1}$, then
\[
\SSS_\bfst( t_i^a t_{i+1}^b \cdot f )  =  
	\SSS_\bfst (- t_i^{b-1} t_{i+1}^{a+1} \cdot f ).
\]
\end{lem}

\begin{proof}
When $f=1$, we note that the result follows from interchanging rows in the Jacobi--Trudi determinant \eqref{eqn:Jacobi.Trudi}. Otherwise, the result follows from Remark \ref{rem:low.deg.of.GG} and Corollary \ref{cor:GG.Straighten.symmetric}. 
\end{proof}

Observe that when $b=a+1$, then Lemma \ref{lem:SSS.JT} implies that
\begin{equation}
\label{eqn:SSS.vanish}
\SSS_\bfst( t_i^a t_{i+1}^{a+1} \cdot f ) = 0,
\end{equation}
a fact we will use in the sequel. In the remainder of the paper, we adopt the notation $(1-t)^J := \prod_{i=1}^k (1-t_i)^{J_i}$ for an integer sequence $J = (J_1,\ldots,J_k)$.

\begin{lem}
\label{lem:SS.expand}
Let $k\geq p$ and $\lambda$ a partition with $\ell(\lambda)=p$. Suppose that $I = (I_1,\ldots,I_k)$ is an integer sequence such that $I_{j} = I_{j-1} + 1$ for some $1\leq j \leq k$.
\begin{enumerate}[label=\emph{(\roman*)},leftmargin=*]
\item If $\lambda_{j-1} = \lambda_{j}$, then
\[
\SSS_\bfst \left( t^\lambda (1-t)^I \right) =  
	\SSS_\bfst  \left( t^\lambda (1-t)^{I-e_j} \right);
\]
\item if $\lambda_{j-1} > \lambda_{j}$ then
\[
\SSS_\bfst \left( t^\lambda (1-t)^{I} \right) =
	\SSS_\bfst \left( 
		t^\lambda (1-t)^{I-e_j}
	- t^{\lambda+e_j} (1-t)^{I-e_j} \right);
\]
\end{enumerate}
where $e_j$ denotes the integer sequence with $1$ in the $j$-th entry and zeroes elsewhere. Furthermore, $I-e_j$ is also a weakly increasing sequence and, in case \emph{(ii)}, $\lambda + e_j$ is still a partition.
\end{lem}

\begin{proof}
For notational simplicity we provide the proof for the case that $k=p=2$; the general proof is analogous. Write $I = (c,c+1)$ and $\lambda = (a,b)$. Thus $I-e_2 = (c,c)$ and $\lambda+e_2 = (a,b+1)$. We have
\begin{multline*}
t^\lambda (1-t)^I = t_1^a t_2^b (1-t_1)^c(1-t_2)^{c+1} 
	= t_1^a t_2^b (1-t_1)^p (1-t_2)^p - t_1^a t_2^{b+1} (1-t_1)^p(1-t_2)^p \\ = t^{\lambda}(1-t)^{I-e_2} - t^{\lambda+e_2}(1-t)^{I-e_2}.
\end{multline*}
Apply the operation $\SSS_{\bfs{t}}$ to both sides above and observe that in the case $a=b$, the second term of the last expression vanishes by \eqref{eqn:SSS.vanish}. In the case $a>b$, this term remains.
\end{proof}

We are now ready to establish that the Schur expansion of a stable Grothendieck polynomial alternates as desired. Recall that for a partition $\lambda$ of length $p$, $k\geq p$, $I=(0,1,2,\ldots,k-1)$, and $\bfs{x} = (x_1,\ldots,x_k)$, Theorem \ref{thm:GG.to.SSS} asserts that
\begin{equation}
\label{eqn:GG.SS.recall}
\GG_{\bfs{t}} \left(t^\lambda\right)
=
\SSS_{\bfs{t}} \left(t^\lambda (1-t)^I\right).
\end{equation}

\begin{proof}[Proof of Theorem \ref{thm:GG.to.SS.alt}]
First observe that the polynomial inside the $\SSS_{\bfs{t}} $ operation of \eqref{eqn:GG.SS.recall} satisfies the hypothesis of Lemma \ref{lem:SS.expand} for the weakly increasing sequence $I = (0,1,2,\ldots,k-1)$.

For each integer $2\leq m\leq k$ define $\mathbf{j}(m) = (m, m+1, \ldots, k)$ and form the sequence
\[
\mathbf{j}= \mathbf{j}(k), \mathbf{j}({k-1}), \mathbf{j}({k-2}), \ldots, \mathbf{j}(2).
\]
Now we will repeatedly apply Lemma \ref{lem:SS.expand} by successively choosing $j$ from the sequence $\mathbf{j}$. To begin, we apply the lemma with $j=k=\mathbf{j}_1$ to $t^\lambda (1-t)^I$ to obtain the sum
\[
t^\lambda (1-t)^{I-e_k} - \delta_{\lambda_{k},\lambda_{k-1}}\cdot t^{\lambda + e_k} (1-t)^{I-e_k}
\]
The lemma can again be applied to both terms for $j=k-1 = \mathbf{j}_2$. Similarly, suppose we are give a given a sum of terms $t^{\lambda'}(1-t)^{I'}$ with $I'_{\mathbf{j}_w} = I'_{\mathbf{j}_{w-1}}+1$ and $\lambda'$ a partition. Then applying Lemma \ref{lem:SS.expand} with $j=\mathbf{j}_w$ gives a sum of polynomials
\[
t^{\lambda'}(1-t)^{I'-e_{\mathbf{j}_w}} - \delta_{\lambda_{\mathbf{j}_w},\lambda_{\mathbf{j}_w-1}}\cdot 
	t^{\lambda'+e_{\mathbf{j}_w}}(1-t)^{I'-e_{\mathbf{j}_w}}
\]
where we observe that both terms above will satisfy the hypotheses for Lemma \ref{lem:SS.expand} with $j=\mathbf{j}_{w+1}$ (notice we have carefully chosen the order for the sequence $\mathbf{j}$ to ensure this is true). We inductively continue this process until $|I'| = 0$ to obtain a sum of monomials $b_\mu t^{\mu}$ for partitions $\mu$. Indeed, notice that each $2\leq j \leq k$ appears $j-1 = I_j$ times in $\mathbf{j}$.

Hence, by our process each $\mu$ with $b_\mu\neq 0$ is of the form $\lambda + \sum_{j=1}^k B_j e_j$ for some integers $0 \leq B_j \leq j-1$. This implies, per case (ii) of Lemma \ref{lem:SS.expand}, that the sign of $b_\mu$ must be $(-1)^{\sum_{j=1}^k B_j} = (-1)^{|\mu|-|\lambda|}$ as desired. Moreover, the fact that every such  $\mu$ is contained in the interval $\lambda \subseteq \mu \subseteq \hat\lambda$ follows from the fact that $0\leq B_j \leq j-1$.
\end{proof}

\begin{example}
Consider the case $\lambda= (2)$ and $k=3$ so that $\bfs{x}=(x_1,x_2,x_3)$. Then the Schur expansion of $G_{(2)}(\bfs{x})$ can be computed by forming the sequences $\mathbf{j}(3)=(3)$, $\mathbf{j}(2) = (2,3)$, and hence $\mathbf{j}=(3,2,3)$. We then transform $t_1^2 (1-t_2)(1-t_3)^2$ according to the diagram below
\[
\begin{diagram}
t_1^2 (1-t_2)(1-t_3)^2 & \rTo^{\mathbf{j}_1=3} & t_1^2(1-t_2)(1-t_3)
 & \rTo^{\mathbf{j}_2=2} & t_1^2(1-t_3) - t_1^2t_2(1-t_3)
 & \rTo^{\mathbf{j}_3=3} & [t_1^2] - [t_1^2 t_2 - t_1^2 t_2 t_3].
\end{diagram}
\]
Notice that each monomial above corresponds to a Schur function $s_\lambda$ for a \emph{partition} once we apply the operation $\SSS_{\bfs{t}}$. In the end, we conclude that 
\[
G_{(2)}(\bfs{x}) = s_{(2)}(\bfs{x}) - s_{(2,1)}(\bfs{x}) + s_{(2,1,1)}(\bfs{x})
\]
and finally, observe that $(2,1,1)$ is indeed the partition $\hat{\lambda}$ corresponding to $\lambda = (2)$ and $k=3$.
\end{example}

\bibliographystyle{amsalpha}
\bibliography{jmabib}

\end{document}